\documentclass[11pt]{amsart}

\usepackage[english]{babel}
\usepackage[T1]{fontenc}
\usepackage[utf8]{inputenc}
\usepackage{amsmath}
\usepackage{amssymb}
\usepackage{amsthm}
\usepackage{color}
\usepackage{tikz}
\usepackage{mathdots}
\usepackage[all]{xy}
\usepackage{hyperref}
\usepackage{accents}
\usepackage{lmodern}
\usepackage{mathrsfs}
\usepackage{appendix}

\theoremstyle{definition}
\newtheorem{definition}{Definition}[section]

\newtheorem{rmk}[definition]{Remark}

\theoremstyle{plain}
\newtheorem{theorem}[definition]{Theorem}
\newtheorem{prop}[definition]{Proposition}

\newtheorem{lem}[definition]{Lemma}
\newtheorem{prob}[definition]{Problem}

\newcommand{\mb}{\mathbb}

\newcommand{\mf}{\mathfrak}

\newcommand{\bs}{\boldsymbol}

\newcommand{\Det}{\textup{Det}}

\begin{document}

\title{Sums of reciprocals of fractional parts II}

\author[Reynold Fregoli]{Reynold Fregoli\vspace{4mm}\\
\MakeLowercase{with an appendix by} \\ Michael Bj\"orklund, Reynold Fregoli, \MakeLowercase{and} Alexander Gorodnik\vspace{2mm}}

\address{MB: Department of Mathematics, Chalmers, Gothenburg, Sweden; RF\,\&\,AG: Department of Mathematics, University of Z\"urich, Switzerland}

\thanks{MB was supported by Swedish VR-grant 11253320,  RF and  AG were supported by SNF grant 200021--182089.}


\dedicatory{}


\begin{abstract}
We prove an estimate for the number of lattice points lying in certain non-convex Euclidean domains of interest in Diophantine approximation. As an application, we generalise a result of Kruse (1964) concerning the almost sure order of magnitude of sums of reciprocals of fractional parts and solve a conjecture posed by Beresnevich, Haynes, and Velani. The methods are based both on the geometry of numbers and on probability theory.
\end{abstract}

\maketitle

\section{Introduction}

\subsection{Counting Points and Diophantine Properties of Lattices}

Let $\varepsilon>0$, $T\geq 1$, and $(\alpha_1,\alpha_2)\in\mb R^2$. Analysing the number of integer solutions $(p_1,p_2,q)$ to the system of inequalities
$$
\begin{cases}
|q\alpha_1+p_1|\cdot|q\alpha_2+p_2|<\varepsilon,\\[1mm] |q\alpha_i+p_1|\leq 1/2 \quad i=1,2,\\[1mm]
0<q<T    
\end{cases}
$$
is naturally relevant to a variety of number-theoretic applications. This is especially true in the field of Diophantine approximation \cite{Wid17},\cite{Fre21-2},\cite[Section 3]{CT19}, where the magnitude of products of rational approximations is the central object of a major unsolved problem: the Littelwood Conjecture (see \cite{Bug14}).

Generalizing the above setup, we let $\bs{L}\in\mb{R}^{m\times n}$ be a matrix with real coefficients, $\bs{T}\in[1,+\infty)^{n}$, and $R,\varepsilon>0$. For a choice of these parameters, we define
\begin{multline}
\label{eq:M}
M(\bs{L},\varepsilon,R,\bs{T}):=\left\{(\bs{p},\bs{q})\in\mb{Z}^{m}\times\left(\mb{Z}^{n}\setminus\{\bs{0}\}\right):\prod_{i=1}^{m}\left|L_{i}\bs{q}+p_{i}\right|<\varepsilon,\right.\\
|L_{i}\bs{q}+p_{i}|\leq R\ i=1,\dotsc,m,\ \ |q_{j}|\leq T_{j}\ j=1,\dotsc,n\Bigg\},
\end{multline}
where $L_i$ for $i=1,\dotsc,m$ denote the rows of the matrix $\bs L$. The primary objective of this paper is to prove general upper bounds for the cardinality of the set $M(\bs{L},\varepsilon,R,\bs{T})$, depending on the properties of the matrix $\bs L$. Our main result carries a number of applications to Diophantine approximation, which we aim to illustrate in the second part of this introduction. 

The problem of bounding $\# M(\bs{L},\varepsilon,R,\bs{T})$ may naturally be recast in terms of lattices. Our goal will be, in fact, to count the number of points in the intersection of the lattice
$$\Lambda_{\bs L}:=\begin{pmatrix}
I_m & \bs L \\
\bs 0 & I_n
\end{pmatrix}\mb Z^{m+n}$$
\noindent with the domain
\begin{equation}
\label{eq:vol}
\left\{|x_1\dotsm x_m|<\varepsilon,\ |x_{1}|,\dotsc,|x_m|\leq R\right\}\times[-T_1,T_1]\times\dotsb \times[-T_n,T_n],
\end{equation}
with the additional restriction that $(x_{m+1},\dotsc,x_{m+n})\neq \bs 0$. This problem is hard to solve for a general lattice, due to the fact that (\ref{eq:vol}) is \emph{not} a convex set. Nonetheless, we shall show that if the matrix $\bs L$ satisfies some appropriate Diophantine conditions, it is possible to obtain non-trivial asymptotic estimates.

For $x\in\mb R$ let $\|x\|$ denote the distance form $x$ to the nearest integer and set $x^{+}:=\max\{|x|,1\}$.

\begin{definition}
\label{def:multbad}
Given a non-increasing function $\phi:[0,+\infty)\to (0,1]$, we say that a matrix $\bs{L}\in\mb{R}^{m\times n}$ is $\phi$-multiplicatively-badly-approximable if the inequality
\begin{equation}
\label{eq:condition1}
q_1^{+}\dotsm q_{n}^{+}\prod_{i=1}^{m}\|L_{i}\bs{q}\|\geq\phi\left(\left(q_1^{+}\dotsm q_{n}^{+}\right)^{1/n}\right)
\end{equation}
holds for all $\bs{q}\in\mb{Z}^{n}\setminus\{\bs{0}\}$.
\end{definition}
\noindent

The assumption in (\ref{eq:condition1}) is in fact quite natural: it may be regarded as a weaker version of admissibility for the lattice $\Lambda_{\bs L}$ with respect to regions of the form
$$\{(x_1,\dotsc,x_{m+n})\in\mb R^{m+n}:|x_1\dotsm x_{m+n}|\leq \phi(|x_1\dotsm x_{m+n}|)\}$$
(see, e.g., \cite[Section III.5.2]{Cas97}). The main difference with the notion of admissibility is as follows: in an admissible lattice, the product of \emph{all} the components of  any non-null lattice vector is required to be bounded below, while here, we allow the last $n$ components of any such vector to attain the value $0$ (note the exponent $+$ in (\ref{eq:condition1})). Admissibility is a typical property of lattices induced by number fields \cite{Skr90} and is known to hold for almost-every unimodular lattice with respect to the Haar measure on the moduli space (provided the function $\phi$ decays sufficiently fast) \cite[Theorem]{KM99}. A weaker version of admissibility was also considered in \cite{Wid18} and \cite{Fre21-2}, once again in the context of Diophantine approximation.

With Definition \ref{def:multbad} at hand, we are now in the position to state our main result. 


\begin{theorem}
\label{prop:cor1}
Let $\bs{L}\in\mb{R}^{m\times n}$ be a $\phi$-multiplicatively badly approximable matrix and suppose that $R^{m}/\varepsilon\geq e^{m}$, where $e=2.71828\dots$ is the base of the natural logarithm. Then for all $\bs T\in[1,+\infty)^{n}$ we have that
\begin{equation}
\label{eq:cor1}
\# M(\bs{L},\varepsilon,R,\bs{T})
\ll_{m,n}(1+R)^{m+n-1}\log\left(\frac{R^{m}}{\varepsilon}\right)^{m-1}\left[\varepsilon \bar T^{n}+\left(\frac{\varepsilon\bar T^{n}}{\phi\left(\bar T\right)}\right)^{\frac{m+n-1}{m+n}}\right],
\end{equation}
where $\bar T:=\left(T_{1}\dotsm T_{n}\right)^{1/n}$.
\end{theorem}
\noindent Here and hereafter, we use the notation $x\ll_z y$ (resp. $x\gg_z y$) to indicate that there exists a constant $c$, depending on the real number $z$, such that $x\leq c\cdot y$ (resp. $x\geq c\cdot y$). The symbol $\asymp$ is used to mean that both these conditions apply.

We remark that the most important feature of Theorem \ref{prop:cor1} is the fact that the upper bound in (\ref{eq:cor1}) depends solely on the geometric mean of the components $T_1,\dotsc, T_n$ of the vector $\bs T$. Indeed, the analogous result with $\bar T$ replaced by $\max_i |T_i|$ would be trivial to prove and is not relevant to applications. 

\subsection{Sums of Reciprocals of Fractional Parts}
Recall that for any $x\in \mb R$, the symbol $\|x\|$ stands for the distance from $x$ to the nearest integer and that $x^{+}=\max\{|x|,1\}$. Fix $n\geq 1$, let $\bs \alpha\in\mb R^{n}$ be a vector such that the numbers $1,\alpha_1,\dotsc,\alpha_n$ are linearly independent over $\mb Z$, and consider the functions
\begin{equation}
\label{eq:s}
S(\bs\alpha, \bs T):=\sum_{\substack{0\leq q_i\leq T_i \\ i=1,\dotsc,n}}\frac{1}{q_1^{+}\dotsm q_n^{+}\|\bs q\cdot\bs\alpha\|}
\end{equation}
 for $\bs T\in[1,+\infty)^n$, and
\begin{equation}
\label{eq:s*}
S^{*}(\bs\alpha, T):=\sum_{0< q\leq T}\frac{1}{q\|q\alpha_1\|\dotsm\|q\alpha_n\|}
\end{equation}
for $T\geq 1$. Here, the vector $\bs q=(q_1,\dotsc,q_n)$ and the number $q$ are assumed to be non-null integers, and the notation $\bs x\cdot\bs y$ stands for the usual dot-product in $\mb R^n$.

Due (not exclusively) to their intimate connection to the theory of uniform distribution, the functions $S(\bs\alpha,\bs T)$ and $S^{*}(\bs\alpha,T)$ have been the object of extensive investigation (see \cite{BHV20} for a complete account). Motivated by this, our main concern in the sequel of this paper will be the following:

\begin{prob}
\label{prob:1}
Determine the growth rate of the sums (\ref{eq:s}) and (\ref{eq:s*}) for a Lebesgue-generic vector $\bs \alpha\in\mb R^n$.    
\end{prob}
\noindent Let us give a brief account of the history of Problem \ref{prob:1}.

The case $n=1$ is relatively well-understood. The function $S(\alpha,T)=S^{*}(\alpha,T)$ was first studied by Hardy and Littlewood in the case when $\alpha$ is a badly approximable number \cite[Memoir II]{HL22}. Problem \ref{prob:1} itself was then addressed by a number of authors \cite{Wal31}, \cite{Sch64}, and finally settled in the case $n=1$ by Kruse \cite[Theorem 6(b)]{Kru64}. Kruse proved that for almost every $\alpha\in\mb R$ one has that
\begin{equation}
\label{eq:kru}
S_{\alpha}(T)\asymp_{\alpha}(\log T)^{2},
\end{equation}
i.e., that for a generic number $\alpha$ upper and lower bounds coincide. Even more precise results involving the partial qutients of the number $\alpha$ (see \cite{BL17}) were established thereafter. We direct the reader to the extensive treatise of Beresnevich, Haynes, and Velani \cite{BHV20} for more details on this.

For $n>1$ the picture is significantly different. In particular, attempts to prove lower bounds as opposed to upper bounds for the sums (\ref{eq:s}) and (\ref{eq:s*}) have been far more successful. Let us recall the main milestones. General lower bounds for the sum $S(\bs\alpha,\bs T)$ were obtained by Beresnevich, Haynes, and Velani, via the Minkowski Convex-Body Theorem \cite{BHV20} (see Remark 1.4 and Theorem 1.10). Their result shows that for any vector $\bs\alpha\in\mb R^n$ and all $\bs T\in[1,+\infty)^n$ it holds that
\begin{equation}
\label{eq:BHVlower}
S(\bs\alpha,\bs T)\gg_{n} \log\bar T\log T_1\dotsm\log T_n,
\end{equation}
where $\bar T$ stands for the geometric mean of the components $T_1,\dotsc,T_n$ of $\bs T$ and $\log x$ denotes the function $\log\max\{ x,e\}$. Lower bounds for the function $S^{*}(\bs \alpha,T)$ are not explicitly present in the literature, but may easily be deduced from the work of L\^e and Vaaler \cite{LV15}, who studied a modification of the sums in (\ref{eq:s}) and (\ref{eq:s*}), where the factors $q_1^+\dotsm q_{n}^+$ and $q$ at the denominator are removed.
Namely, employing techniques from harmonic analysis, L\^e and Vaaler \cite[Theorem 1.1]{LV15} show that for all $\bs\alpha\in\mb R^n$ and all $T\geq 1$ one has that
\begin{equation}
\label{eq:LV}
\sum_{q=1}^{T}\frac{1}{\|q\alpha_1\|\dotsm\|q\alpha_n\|}\gg_{n}T(\log T)^{n+1}.    
\end{equation}
Note that (\ref{eq:LV}), along with the the Abel Summation Formula
$$\sum_{n=1}^{N}a_n\cdot b_n=\sum_{n=1}^{N}(a_1+\dotsb +a_n)(b_{n+1}-b_{n})+(a_1+\dotsb +a_{N})b_{N+1},$$
for sequences of real numbers $\{a_n\},\{b_n\}$, easily yields
\begin{equation}
\label{:BHV*lower}
S^{*}(\bs\alpha,T)\gg_{n}(\log T)^{n+1},
\end{equation}
\noindent giving a lower bound akin to (\ref{eq:BHVlower}) for the function $S^{*}(\bs\alpha,T)$.

In stark contrast to the above results, sharp upper bounds for the sums $S(\bs\alpha,\bs T)$ and $S^{*}(\bs\alpha,T)$ and generic $\bs \alpha\in\mb R^n$ ($n>1$) are not presently available. Schmidt \cite[Theorem 2]{Sch64} proved that, in the symmetric case (i.e., when $T_1=\dotsb =T_n=T$), the inequality
$$S(\bs\alpha,\bs T)\ll_{n,\bs \alpha} (\log T)^{n+1+\varepsilon}$$
holds for almost every $\bs \alpha\in\mb R^n$. He additionally derived an extension of this result to arbitrary sequences of integers, which he then applied to give bounds for the discrepancy of multi-dimensional Kronecker sequences \cite[Theorem 3]{Sch64}. A second partial result towards (\ref{eq:BHVmainbody}) is due to Beck, once again in the context of uniform distribution theory \cite[Lemma 4.1]{Bec94}. Beck showed that, in the symmetric case, for almost every $\bs\alpha\in\mb R^n$ and all $T\geq 1$ it holds that
$$\sum_{\substack{0\leq q_i\leq T \\ i=1,\dotsc,n \\ q_1^+\dotsm q_n^+\|\bs\alpha\cdot\bs q\|\leq (\log T)^{20n}}}\frac{1}{q_1^{+}\dotsm q_n^{+}\|\bs q\cdot\bs\alpha\|}\ll_{n,\bs\alpha}\varphi(\log\log T)(\log T)^{n},$$
where $\varphi:\mb N\to\mb N$ is any non-decreasing function such that $\sum_{n}\varphi(n)^{-1}<+\infty.$
In the same spirit, estimates for the sums $S(\bs\alpha,\bs T)$ were used in \cite{HL12} to bound the discrepancy function for Halton-Kronecker sequences. All these results however, do not seem to easily extend to the non-symmetric case.

 Based on (\ref{eq:kru}) and (\ref{eq:BHVlower}), Beresnevich, Haynes, and Velani \cite[Conjecture 1.1]{BHV20} conjectured that for almost every $\bs\alpha\in\mb R^n$ and all $\bs T\in[1,+\infty)^n$ it holds that
\begin{equation}
\label{eq:BHVmainbody}
 S(\bs\alpha,\bs T)\ll_{\bs{\alpha},n}\log \bar T\log T_{1}\dotsm \log T_{n}.   
\end{equation}
By duality, we might also expect that for almost every $\bs\alpha\in\mb R^n$ and all $T\geq 1$ the analogous inequality
\begin{equation}
\label{eq:dualBHVmainbody}
S^{*}(\bs\alpha,T)\ll_{\bs{\alpha},n}(\log T)^{n+1}
\end{equation}
holds true.

Our first application of Theorem \ref{prop:cor1} is an unconditional proof of (\ref{eq:BHVmainbody}) and (\ref{eq:dualBHVmainbody}), which allows us to fully settle Problem \ref{prob:1}.

\begin{theorem}
\label{thm:BHV}
For almost every $\bs{\alpha}=(\alpha_{1},\dotsc,\alpha_{n})\in\mb{R}^{n}$ and all $\bs T\in[1,+\infty)^{n}$ it holds that
 \begin{equation}
 \label{eq:BHV}
S(\bs\alpha,\bs T)\asymp_{\bs{\alpha},n}\log \bar T\log T_{1}\dotsm \log T_{n},   
 \end{equation}
where $\bar{T}$ denotes the product $(T_{1}\dotsm T_{n})^{1/n}$. Analogously, for almost every $\bs \alpha\in\mb R^n$ and all $T\geq 1$ it holds that 
\begin{equation}
\label{eq:BHV*}
 S^{*}(\bs\alpha,T)\asymp_{\bs{\alpha},n}(\log T)^{n+1}.   
\end{equation}
\end{theorem}

\medskip

As a second application of Theorem \ref{prop:cor1}, we are able to prove upper bounds for non-averaged sums, such as (\ref{eq:LV}). This complements lower bounds of L\^e and Vaaler \cite[Corollary 1.2]{LV15} and significantly strengthens their upper bounds \cite[Theorem 2.1]{LV15}, by allowing $\phi$ in (\ref{eq:condition1}) to be a generic non-increasing function rather than a fixed constant. Define
$$\Sigma(\bs L,\bs T):=\sum_{\substack{0\leq q_i\leq T_i \\ i=1,\dotsc,n}}\frac{1}{\|\bs q\cdot L_1\|\dotsm \|\bs q\cdot L_m\|}$$
for $\bs L\in\mb R^{m\times n}$. The following result then holds.

\begin{theorem}
\label{thm:LV}
Let $\phi:[0,+\infty)\to (0,1]$ be a non-increasing real function and let $\bs{L}\in\mb{R}^{m\times n}$ be a $\phi$-multiplicatively badly approximable matrix. Then for all $\bs{T}\in[1,+\infty)^{n}$ it holds that
\begin{equation}
\Sigma(\bs L,\bs T)\ll_{m,n}\bar T^{n}\log\left(\frac{\bar T}{\phi(\bar T)}\right)^{m}+\frac{\bar T^{n}}{\phi(\bar T)}\log\left(\frac{\bar T}{\phi(\bar T)}\right)^{m-1},\nonumber
\end{equation}
provided that $\bar T:=\left(T_{1}\dotsm T_{n}\right)^{1/n}\geq 2$.
\end{theorem}

The proof of Theorem \ref{thm:LV} will be the object of Appendix \ref{sec:proofofthmLV}. We conclude our introduction by remarking that, in the non-averaged setting (i.e., without the factors $q_1,\dots,q_n$ in the denominator), an analogue of Theorem \ref{thm:BHV} seems to be far-fetched. In fact, based on current evidence, the behaviour of the functions $\Sigma(\bs L,\bs T)$ should be much more erratic. The reader is directed to the discussion in \cite[Section 1.3.2]{CT19} for more details on this. 

\subsection{Methods}

The main idea behind the proof of Theorem \ref{prop:cor1} is to tessellate the counting domain (e.g., (\ref{eq:vol})) into volume-$1$ tiles by using diagonal maps. Each tile can then be moved into a cube centred (essentially) at the origin, through the action of the corresponding map. The counting estimate follows from studying the successive minima of all the images of the lattice $\Lambda_{\bs\alpha^{t}}$ under these diagonal maps. Since the maps in question are diagonal, volume-preserving, and (\ref{eq:condition1}) holds, we may derive an estimate on the minima in terms of the function $\phi$, by simply applying the arithmetic-geometric mean inequality. This will have to be done carefully enough to account for the \emph{null} components of any vector realising a successive minimum. The main novelty of this paper lies precisely in this argument, which was inspired by previous work of Widmer \cite{Wid17}. Details may be found in Sections \ref{sec:proofM} and \ref{sec:proofM1}.

\medskip

The proof of Theorem \ref{thm:BHV} relies on two separate arguments. One is based on techniques from the geometry of numbers (i.e., Theorem \ref{prop:cor1}), while the other is more analytic in nature and is inspired by the work of Schmidt and Cassels.

We briefly illustrate the main features of both arguments for the case of the function $S(\bs\alpha,\bs T)$. To start, for $0<a<b$ we introduce the counting function
$$N(\bs{\alpha},\bs T,a,b):=\#\left\{\bs q\in[0,T_{1}]\times\dotsm\times[0,T_{n}]\cap\mb{Z}^{n}:a<q_{1}\dotsm q_{n}\|\bs q\cdot\bs\alpha\|\leq b\right\}.$$
Then, excluding vectors $\bs q\in\mb Z^n$ with at least one null component (whose contribution is estimated inductively), we have that
\begin{equation}
\label{eq:newexplanation}
S(\bs\alpha,\bs T)\leq \sum_{k=-\log\left (T_1\dotsm T_n\right)}^{+\infty}e^{k+1}N(\bs{\alpha},\bs T,e^{-k-1},e^{-k}).
\end{equation}
The problem of establishing upper bounds for the function $S(\bs\alpha,\bs T)$ is therefore reduced to estimating the function $N(\bs{\alpha},\bs T,e^{-k-1},e^{-k})$ for $k$ larger than $-\log (T_1\dotsm T_n)$. By considering dyadic intervals for the variables $q_i$, we can further reduce ourselves to analysing the quantity
\begin{multline}
\label{eq:moredyadic}
\#\bigg\{\bs q\in\left[e^{h_1-1},e^{h_1}\right]\times\dotsm\times\left[e^{h_n-1},e^{h_n}\right]\cap\mb{Z}^{n}: \\
e^{-k-h_1-\dotsb -h_{n}-1}<\|\bs q\cdot\bs\alpha\|\leq e^{-k-h_{1}-\dotsm -h_n+1}\bigg\}
\end{multline}
for $0\leq h_i\leq \log T_i$. 
This may be done by using (\ref{eq:cor1}). In particular, if the function $\phi$ decays sufficiently fast, the set of $\phi$-multiplicatively badly approximable matrices $\bs\alpha^t$ will have full Lebesgue measure, allowing us to use Theorem \ref{prop:cor1} for almost all $\bs\alpha\in[0,1)^n$.

The strategy described in the previous paragraph is fruitful so long as the main term $\varepsilon\bar T^{n}$ in (\ref{eq:cor1}) is larger than $1$ (note that $\varepsilon=e^{-k-h_1-\dotsb -h_n}$ and $\bar T^n=e^{h_1+\dotsb +h_n}$). If this is not the case, the error term exceeds the main term and the method brakes down. Hence, the necessity to distinguish two different ranges for the parameter $k$ in (\ref{eq:newexplanation}) .
When $k$ is too large, we use a different argument, based on ideas of Schmidt \cite{Sch60}, which can, in turn, be traced back to Cassels \cite{Cas50}. This argument is, in fact, fairly general. Consider a sequence of functions $f_{\bs n}(\bs\alpha)$ which, in our case, are defined as the size of the intersection of $\Lambda_{\bs\alpha^t}$ with the set in (\ref{eq:vol}). Our goal is to obtain an almost-sure estimate for the function $\sum_{n_i\leq N_i}f_{\bs n}(\bs\alpha)$. Denote by $m_{\bs n}$ the first moment of each function $f_{\bs n}(\bs\alpha)$. Then it is easy to see that $\sum_{n_i\leq N_i}m_{\bs n}$ coincides up to multiplication by a constant with the volume term $\varepsilon \bar T^n$ in (\ref{eq:cor1}), when $N_i=\log T_i$. This shows that the sum $\sum_{n_{i}\leq N_i}m_{\bs n}$ gives a good upper bound for the purpose of proving Theorem \ref{thm:BHV}. We then observe that for each choice of the integers $N_i$ and of intervals $0<A_i\leq B_i\leq N_i$ one has that
\begin{equation}
\label{eq:Schmidtexpl}
\int_{[0,1)^n} \left|\sum_{A_i<n_i\leq B_i}f_{\bs n}(\bs\alpha)\right|d\bs\alpha\ll (B_1-A_1)\dotsm(A_n-B_n)\max_{n_{i}\leq N_i}m_{\bs n}.
\end{equation}
Combining (\ref{eq:Schmidtexpl}) with the Borel-Cantelli Lemma and a dyadic argument similar to that in \cite{Cas50}, one shows that for almost every $\bs\alpha\in [0,1)^n$ and for all but finitely many integers $N_i$ the sums $\sum_{n_i\leq N_i}f_{\bs n}(\alpha)$ is bounded above by $\sum_{n_i\leq N_i}m_{\bs n}$, up to multiplication by a power of $\log(N_1\dotsm N_n)$. The presence of this logarithmic factor (worsening the estimate given by the expected value) is compensated by the fact that the range of different parameters $k$ for which this argument is applied is relatively small. In light of this, we are still able to derive a favourable global upper bound. This is the object of Appendix \ref{sec:fmm}, written in collaboration with Michael Bj\"orklund and Alexander Gorodnik.

\section{Tessellations}
\label{sec:tess}

In this section we present two covering results for sets of the form 
\begin{equation}
\label{eq:asabove}
\{\bs x\in\mb R^d:|x_1\dotsm x_d|<\varepsilon\}.    
\end{equation}
The first will be used to prove both Theorem \ref{prop:cor1} and Theorem \ref{thm:BHV}, while the second will only appear in the proof of Theorem \ref{thm:BHV}. The essence of both results is that a set as in \eqref{eq:asabove}, with given bounds on the variables $x_i$, may be tessellated via the diagonal action of $\textup{SL}_d(\mb R)$ into a controlled number of volume-$1$ tiles. 

Let
\begin{equation*}
H_1:=\left\{\bs{x}\in\mb{R}^{{m}}:\prod_{i= 1}^{{m}}\left|x_{i}\right|<\varepsilon,\ |x_{i}|\leq R,\ i=1,\dotsc,{m}\right\}
\end{equation*}
and
$$H_{1+}:=H_1\cap\left\{\bs{x}\in\mb{R}^{{m}}:x_{i}\neq 0,\ i=1,\dotsc,{m}\right\}.$$

\begin{prop}
\label{thm:partition}
Suppose that $R^{{m}}/\varepsilon>e^{{m}}$, where $e=2.71828\dots$ is the base of the natural logarithm. Then there exist a set of indices $I$, a partition $H_{1+}=\bigcup_{\beta\in{I}}\! X_{\beta}$ of the set $H_{1+}$, and a collection of linear maps $\left\{\varphi_{\beta}\right\}_{\beta\in{{I}}}$ from $\mb{R}^{{m}}$ to itself, such that\vspace{2mm}
\begin{itemize}
\item[$i)$] $\#{I}\ll_{{m}}\log\left(R/\varepsilon^{1/{m}}\right)^{{m}-1}$;\vspace{2mm}
\item[$ii)$] the maps $\varphi_{\beta}$ for $\beta\in{I}$ are determined by the expression $\varphi_{\beta}(\bs{x})_{i}:=e^{a_{\beta,i}}\cdot x_{i}$ for
$i=1,\dotsc,{m}$, where the coefficients $a_{\beta,i}\in\mb{R}$ satisfy\vspace{2mm}
\begin{itemize}
\item[$iia)$] $e^{a_{\beta,i}}\gg_{{m}}\varepsilon^{1/{m}}/R$ for $i=1,\dotsc,{m}$;\vspace{2mm}
\item[$iib)$] $\sum_{i=1}^{{m}}a_{\beta,i}=0$;
\end{itemize}
\vspace{2mm}
\item[$iii)$] the sets $X_{\beta}$ are measurable and $\varphi_{\beta}\left(X_{\beta}\right)\subset\left[-c\varepsilon^{1/{m}},c\varepsilon^{1/{m}}\right]^{{m}}$ for all $\beta\in{I}$, where $c$ is a constant only depending on $m$.\vspace{2mm}
\end{itemize}
\end{prop}
\noindent The proof of this proposition can be found in \cite{Fre21-2} (see in particular Proposition 2.1 with $\bs\beta=(1,\dots,1)$ and $\varepsilon$ replaced by $\varepsilon^{1/m}$).

For $\varepsilon>0$, $\bs T\in[1,+\infty)^{n}$, and $0<R\leq 1$ consider the set
$$H_2:=\left\{\bs{x}\in\mb{R}^{n+1}: \prod_{i=0}^{n}|x_i|\leq \varepsilon,\ |x_0|\leq R,\ 1\leq |x_{i}|\leq T_i,\ i=1,\dotsc,n\right\}$$
and define
$$H_{2+}:=H_2\cap\left\{\bs{x}\in\mb{R}^{n+1}:x_{i}\neq 0,\ i=0,\dotsc,{n}\right\}.$$

\begin{prop}
\label{prop:tess}
Let $\varepsilon,\bs T$, and $R$ as above, and assume that $\varepsilon<RT_1\dotsm T_n$. Then there exist a set of indices $J$, a covering $H_{2+}\subset\bigcup_{\beta\in{J}}\! Y_{\beta}$ of the set $H_{2+}$, and a collection of linear maps $\left\{\psi_{\beta}\right\}_{\beta\in{{J}}}$ from $\mb{R}^{{n}+1}$ to itself, such that\vspace{2mm}
\begin{itemize}
    \item[$i)$] $ J=\left([0,\log T_1]\times\dotsb\times[0,\log T_n]\cap\mb Z^{n}\right)\times\{1,\dotsc,2^{n}\}$ (in particular, $J$ is independent of the choice of $\varepsilon$);\vspace{2mm}
    \item[$ii)$] the maps $\psi_{\beta}$ for $\beta\in{J}$ are determined by the expressions $\psi_{\beta}(\bs{x})_{i}:=\pm e^{b_{\beta,i}}\cdot x_{i}$ for $i=0,\dotsc,{n}$ and the coefficients $b_{\beta,i}$ satisfy\vspace{2mm}
    \begin{itemize}
        \item[$iia)$] $b_{\beta,0}\geq 0$ and $b_{\beta,1},\dots,b_{\beta,n}\leq 0$;\vspace{2mm}
        \item[$iib)$] $\sum_{i=0}^{n}b_{\beta,i}=0$;
    \end{itemize}
    \vspace{2mm}
    \item[$iii)$] the sets $Y_\beta$ are measurable and $\psi_\beta(Y_\beta)\subset (0,\varepsilon]\times[1,e]^{n}$ for all $\beta\in J$.
\end{itemize}
\end{prop}

The proof of Proposition \ref{prop:tess} is similar to that of Proposition \ref{thm:partition}, but we report it here for completeness.

\begin{proof}
By symmetry, we may reduce to consider only the set $H_{2+}':=H_{2+}\cap(0,+\infty)^{n}$. Let $\pi:(0,+\infty)^{n}\to\mb{R}^n$ be the map
$$\pi(x,y,z):=(\log (x_0\dotsm x_n),\log x_1,\dotsc,\log x_n)=(s,v_1,\dotsc,v_n).$$
For $\bs b\in\mb{Z}^n$ define the translation $\tau (\bs b):\mb R^{n+1}\to\mb R^{n+1}$ given by
$$\tau(\bs b)(s,v_1,\dotsc,v_n):=(s,v_1+b_1,\dotsc,v_n+b_n).$$
Then, we have
$$\textup{diag}\left(e^{-\sum_{i}b_i},e^{b_1},\dotsc,e^{b_n}\right)=\pi^{-1}\circ \tau(\bs b)\circ \pi,$$
where $\textup{diag}(\bs x)_{ij}:=x_{i}\delta_{ij}$ for $\bs x\in\mb{R}^{n+1}$ and $i,j=0,\dotsc,n$. Moreover, with $J':=[0,\log T_1]\times\dotsb\times[0,\log T_n]\cap\mb Z^n$, it holds that
$$\pi(H_{2+}')\subset\bigcup_{\bs{b}\in J'}\tau(\bs b)\left((-\infty,\log \varepsilon]\times[0,1)^{n}\right).$$
It follows that
$$H_{2+}'\subset\bigcup_{\bs b\in J'}\pi^{-1}\circ\tau(\bs b)\circ\pi\left((0,\varepsilon]\times[1,e]^{n}\right).$$
To conclude, it is enough, to set
$$\psi_{\bs b}:=\pi^{-1}\circ \tau(-\bs b)\circ \pi$$
and
$$Y_{\bs b}:=\psi_{\bs b}^{-1}\left((0,\varepsilon]\times[1,e]^{n}\right)$$
for all $\bs b\in J'$. The indices $J$, the sets $Y_{\beta}$, and the maps $\psi_{\beta}$ are obtained by a similar argument applied to all the remaining orthants of $\mb R^n$.
\end{proof}

\section{Proof of Theorem \ref{prop:cor1}}
\label{sec:proofM}

As discussed in the introduction, estimating the quantity $\#M(\bs{L},\varepsilon,R,\bs{T})$ can be reduced to a lattice-point counting problem. Here and hereafter, we write $\bs{v}=(\bs{x},\bs{y})$, with $\bs{x}\in\mb{R}^{m}$ and $\bs{y}\in\mb{R}^{m}$, for vectors $\bs{v}\in\mb{R}^{m+n}$. We let
\begin{equation}
\label{eq:AL}
A_{\bs{L}}:=\begin{pmatrix}
I_m & \bs{L} \\
\bs{0} & I_{n}
\end{pmatrix}\in\mb{R}^{(m+n)\times(m+n)},    
\end{equation}
where $I_{m}$ and $I_{n}$ are identity matrices of size $m$ and $n$ respectively, and we let $\Lambda_{\bs{L}}:=A_{\bs{L}}\mb{Z}^{m+n}\subset\mb{R}^{m+n}$. We define the sets
\begin{equation*}
H_1:=\left\{\bs{x}\in\mb{R}^{m}:\prod_{i= 1}^{m}\left|x_{i}\right|<\varepsilon,\ |x_{i}|\leq R,\ i=1,\dotsc,m\right\}
\end{equation*}
and $Z:=H_1\times\prod_{j=1}^{n}[-T_{j},T_{j}]\subset\mb{R}^{m+n}$. Then, we have that
\begin{equation}
\label{eq:intersection}
\# M(\bs{L},\varepsilon,R,\bs{T})=\#\left(\left(\Lambda_{\bs{L}}\cap Z\right)\setminus C\right),
\end{equation}
where $C:=\{\bs{y}=\bs{0}\}\subset\mb{R}^{m+n}$. Since $\Lambda_{\bs{L}}\cap C=\mb{Z}^{m}\times\{\bs{0}\}$, to estimate $\# M(\bs{L},\varepsilon,R,\bs{T})$, it suffices to determine the cardinality of the set $\Lambda_{\bs{L}}\cap Z$.

We start by partitioning $Z$ through Proposition \ref{thm:partition}. Let
$$H_{1+}:=H_1\cap\left\{\bs{x}\in\mb{R}^{m}:x_{i}\neq 0,\ i=1,\dotsc,m\right\}$$
and let $Z_{+}:=H_{1+}\times\prod_{j=1}^{n}[-T_{j},T_{j}]$. Let also
$$H_1^{i}:=H_1\cap\{x_{i}=\bs{0}\}$$
and $Z^{i}:=H_1^{i}\times\prod_{j=1}^{n}[-T_{j},T_{j}]$ for $i=1,\dotsc,m$. Then, we have
\begin{equation}
Z=Z_{+}\cup\bigcup_{i=1}^{m}Z^{i}.\nonumber
\end{equation}
It follows that
\begin{equation}
\#(\Lambda_{\bs{L}}\cap Z)\leq\#(\Lambda_{\bs{L}}\cap Z_{+})+\sum_{i=1}^{m}\#\left(\Lambda_{\bs{L}}\cap Z^{i}\right).\nonumber
\end{equation}

Now, we apply Proposition \ref{thm:partition} to the set $H_1$. We set $\tilde{X}_{\beta}:=X_{\beta}\times\prod_{j=1}^{n}[-T_{j},T_{j}]$ for $\beta\in I$, and we extend the maps $\varphi_{\beta}$ to $\tilde{\varphi}_{\beta}:\mb{R}^{m+n}\to\mb{R}^{m+n}$ ($\beta\in I$), by defining $\tilde{\varphi}_{\beta}$ as the identity map on the second $n$ coordinates. In view of this, we find a partition
$$Z_{+}=\bigcup_{\beta\in I}\tilde{X}_{\beta}$$
of the set $Z_{+}$ and we have that
\begin{align}
 \#\left(\Lambda_{\bs{L}}\cap Z\right) & \leq\#(\Lambda_{\bs{L}}\cap Z_{+})+\sum_{i=1}^{m}\#\left(\Lambda_{\bs{L}}\cap Z^{i}\right)\nonumber\\
 & =\sum_{\beta\in I}\#\left(\Lambda_{\bs{L}}\cap \tilde{X}_{\beta}\right)+\sum_{i=1}^{m}\#\left(\Lambda_{\bs{L}}\cap Z^{i}\right)\nonumber \\
 & =\sum_{\beta\in I}\#\left(\tilde{\varphi}_{\beta}(\Lambda_{\bs{L}})\cap \tilde{\varphi}_{\beta}\left(\tilde{X}_{\beta}\right)\right)+\sum_{i=1}^{m}\#\left(\Lambda_{\bs{L}}\cap Z^{i}\right)\nonumber.
\end{align}
\noindent We deal with these terms separately. We start with $\#\left(\Lambda_{\bs{L}}\cap Z^{i}\right)$ for $i=1,\dotsc,{m}$.

\begin{lem}
\label{lem:LambdacapC}
For $i=1,\dotsc,{m}$ we have
\begin{equation}
\#\left(\Lambda_{\bs{L}}\cap Z^{i}\right)\ll_{{m}}(1+R)^{{m}}.\nonumber
\end{equation}
\end{lem}

\begin{proof}
Since the entries $L_{i1},\dotsc,L_{i{n}}$ of the matrix $\bs L$ along with $1$ are linearly independent over $\mb{Z}$, the equation $L_{i}\bs{q}+p_{i}=0$ implies that $\bs{q}=\bs{0}$. It follows that $\Lambda_{\bs{L}}\cap Z^{i}\subset \Lambda_{\bs{L}}\cap C$ for $i=1,\dotsc,{m}$. Hence, we have that $\#(\Lambda_{\bs{L}}\cap Z^{i})\leq\#((\Lambda_{\bs{L}}\cap C)\cap(Z\cap C))$. Now, we observe that $\Lambda_{\bs{L}}\cap C=\mb{Z}^{{m}}\times\{\bs{0}\}$ and $Z\cap C\subset[-R,R]^{{m}}\times\{\bs{0}\}$. This immediately yields the required inequality.
\end{proof}

We are left to estimate the quantity $\#\left(\tilde{\varphi}_{\beta}(\Lambda_{\bs{L}})\cap \tilde{\varphi}_{\beta}\left(\tilde{X}_{\beta}\right)\right)$ for $\beta\in{I}$. To do so, we use a result from \cite{BW14} which generalises a theorem of Davenport (see \cite[Equation (1.2)]{BW14} and discussion therein).
\begin{theorem}
\label{thm:latticeest}
Let $d\in\mb{N}$ and let $\Lambda$ be a full rank lattice in $\mb{R}^{d}$. Let also $Q\geq 1$ and $B_{Q}:=[-Q,Q]^d\subset\mb{R}^{d}$. Then we have that
$$\#(\Lambda\cap B_{Q})\ll_{d}1+\sum_{s=1}^{d}\frac{Q^s}{\delta_{1}\dotsm\delta_{s}},$$
where $\delta_{s}$ is the $s$-th successive minimum of the lattice $\Lambda$.
\end{theorem}
\noindent Here and throughout the rest of the paper, the $s$-th successive minimum of a full-rank lattice $\Lambda\subset\mb R^d$ will stand for
$$\min\{r>0:B(0,r)\cap\Lambda\mbox{ contains }s\mbox{ linearly independent vectors}\},$$
with $B(0,r)$ a Euclidean ball of radius $r$ centered at the origin. Moreover, for simplicity, in the remainder of this section and in the next one we will denote the geometric mean $\bar T$ of the numbers $T_1,\dotsc,T_j$ simply by $T$.

From Proposition \ref{thm:partition} it follows that
\begin{equation}
\label{eq:27}
\tilde{\varphi}_{\beta}\left(\tilde{X}_{\beta}\right)\subset\left[-c\varepsilon^{1/{m}},c\varepsilon^{1/{m}}\right]^{{m}}\times\prod_{j=1}^{{n}}[-T_{j},T_{j}]
\end{equation}
for $\beta\in{I}$, where $c>0$ only depends on $m$. To apply Theorem \ref{thm:latticeest}, we re-scale the box in (\ref{eq:27}) and transform it into a cube. We set
$$\theta:=\frac{\left(\varepsilon T^{{n}}\right)^{\frac{1}{{m}+{n}}}}{\varepsilon^{1/{m}}},$$
and we consider two linear maps $\omega_{1},\omega_{2}:\mb{R}^{{m}+{n}}\to \mb{R}^{{m}+{n}}$, defined by
$$\omega_{1}(\bs{x},\bs{y}):=\left(\bs{x},\frac{T}{T_{1}}y_{1},\dotsc,\frac{T}{T_{{n}}}y_{{n}}\right)$$
and
$$\omega_{2}\left(\bs{x},\bs{y}\right):=\left(\theta\bs{x},\theta^{-\frac{{m}}{{n}}}\bs{y}\right)$$
for $(\bs{x},\bs{y})\in\mb{R}^{{m}+{n}}$.
Then we have that
\begin{equation}
\omega_{2}\circ\omega_{1}\left(\left[-\varepsilon^{1/{m}},\varepsilon^{1/{m}}\right]^{{m}}\times\prod_{j=1}^{{n}}[-T_{j},T_{j}]\right)=\left[-\left(\varepsilon T^{{n}}\right)^{\frac{1}{{m}+{n}}},\left(\varepsilon T^{{n}}\right)^{\frac{1}{{m}+{n}}}\right]^{{m}+{n}}.\nonumber
\end{equation}
\noindent Hence, from Lemma \ref{lem:LambdacapC} and Theorem \ref{thm:latticeest} we deduce that
\begin{align}
\label{neweq:estimate3}
 & \#\left(\Lambda_{\bs{L}}\cap Z\right)\leq\sum_{\beta\in{I}}\#\left(\omega_{2}\circ\omega_{1}\circ\tilde{\varphi}_{\beta}(\Lambda_{\bs{L}})\cap \omega_{2}\circ\omega_{1}\circ\tilde{\varphi}_{\beta}\left(\tilde{X}_{\beta}\right)\right)+\sum_{i=1}^{{m}}\#\left(\Lambda_{\bs{L}}\cap Z^{i}\right)\nonumber\\
 & \ll_{{m},{n}}\sum_{\beta\in{I}}\left(1+\sum_{s=1}^{{m}+{n}}\frac{\left(\varepsilon T^{{n}}\right)^{\frac{s}{{m}+{n}}}}{\delta_{1}\dotsm\delta_{s}}\right)+(1+R)^{{m}},
\end{align}
where $\delta_{s}$ denotes the $s$-th successive minimum of the lattice $\omega_{2}\circ\omega_{1}\circ\tilde{\varphi}_{\beta}(\Lambda_{\bs{L}})$ ($\beta\in{I}$) for $s=1,\dotsc,{m}+{n}$.
We are therefore left to estimate the quantities
$$\frac{\left(\varepsilon T^{{n}}\right)^{\frac{s}{{m}+{n}}}}{\delta_{1}\dotsm\delta_{s}}$$
for $\beta\in{I}$.

\begin{prop}
\label{prop:firstminima}
Let $\beta\in{I}$ and let $\delta_{1},\dotsc,\delta_{{m}+{n}}$ be the successive minima of the lattice
$\omega_{2}\circ\omega_{1}\circ\tilde{\varphi}_{\beta}(\Lambda_{\bs{L}})$. Then
\begin{align}\
\label{eq:firstminima} 
 & \frac{\left(\varepsilon T^{{n}}\right)^{\frac{s}{{m}+{n}}}}{\delta_{1}\dotsm\delta_{s}}\ll_{{m},{n}}1+R^{{m}+{n}-1}+\varepsilon T^{{n}}+\left(\frac{\varepsilon T^{{n}}}{\phi(T)}\right)^{\frac{{m}+{n}-1}{{m}+{n}}}.
\end{align}
for all $s=1,\dotsc,{m}+{n}$.
\end{prop}
The proof of Proposition \ref{prop:firstminima} (which forms the backbone of the whole argument) is postponed to Section \ref{sec:proofofmin}.

Combining (\ref{neweq:estimate3}) and Proposition \ref{prop:firstminima}, we find that
\begin{equation}
\#\left(\Lambda_{\bs{L}}\cap Z\right)\ll_{{m},{n}}\#{I}\left((1+R)^{{m}+{n}-1}+\varepsilon T^{{n}}+\left(\frac{\varepsilon T^{{n}}}{\phi(T)}\right)^{\frac{{m}+{n}-1}{{m}+{n}}}\right).\nonumber
\end{equation}
By (\ref{eq:intersection}), this implies
\begin{equation}
\label{eq:almostlasteq2}
\#M(\bs{L},\varepsilon,R,\bs{T})\ll_{{m},{n}}\#{I}\left((1+R)^{{m}+{n}-1}+\varepsilon T^{{n}}+\left(\frac{\varepsilon T^{{n}}}{\phi(T)}\right)^{\frac{{m}+{n}-1}{{m}+{n}}}\right).
\end{equation}
In view of Proposition \ref{thm:partition}, we also have that $\#{I}\ll_{{m}}\log\left(R/\varepsilon^{1/{m}}\right)^{{m}-1}$. Therefore, if $\varepsilon T^{{n}}/\phi(T)\geq 1$, the required estimate is a straightforward consequence of (\ref{eq:almostlasteq2}). We are then left to prove the claim in the case when $\varepsilon T^{{n}}/\phi(T)<1$. To do this, we rely on the following lemma.

\begin{lem}
\label{lem:emptycase}
Assume that $\varepsilon T^{{n}}/\phi(T)<1$. Then we have that $\Lambda_{\bs{L}}\cap Z\subset C$.
\end{lem}

\begin{proof}
Suppose by contradiction that there exists a vector $\bs{v}\in(\Lambda_{\bs{L}}\cap Z)\setminus C$. Then, we can write
$$\bs{v}=(L_{1}\bs{q}+p_{1},\dotsc,L_{{m}}\bs{q}+p_{{m}},\bs{q})$$
for some $\bs{p}\in\mb{Z}^{{m}}$ and $\bs{q}\in\mb{Z}^{{n}}\setminus\{\bs{0}\}$. However, since $\bs{v}\in Z$, we have that
\begin{equation}
q_1^{+}\dotsm q_n^{+}\prod_{i=1}^{{m}}\left\|L_{i}\bs{q}\right\|\leq q_1^{+}\dotsm q_n^{+}\prod_{i=1}^{{m}}\left|L_{i}\bs{q}+p_{i}\right|\leq T^{{n}}\varepsilon<\phi(T)\leq\phi\left(\left(q_1^{+}\dotsm q_n^{+}\right)^{\frac{1}{{n}}}\right),\nonumber
\end{equation}
in contradiction with (\ref{eq:condition1}). This proves the claim.
\end{proof}
By Lemma \ref{lem:emptycase} and (\ref{eq:intersection}), if $\varepsilon T^{{n}}/\phi(T)< 1$, we have that $M(\bs{L},\varepsilon,R,\bs{T})=\emptyset$, and (\ref{eq:cor1}) becomes trivial. Hence, the proof is complete.

\section{Proof of Proposition \ref{prop:firstminima}}
\label{sec:proofofmin}

\subsection{Construction of a Basis}

The goal of this subsection is to construct a "workable" basis for the lattice $\omega_{2}\circ\omega_{1}\circ\tilde{\varphi}_{\beta}(\Lambda_{\bs{L}})$ in order to prove Proposition \ref{prop:firstminima}. We start by recalling the following general result on lattices.

\begin{theorem}[Mahler-Weyl]
\label{theorem:Mahler-Weyl}
Let $d\geq 1$ and let $\Lambda$ be a full-rank lattice in $\mb{R}^{d}$ with successive minima $\delta_{1},\dotsc,\delta_{d}$. Then there exists a basis $\bs{v}^{1},\dotsc,\bs{v}^{d}$ of $\Lambda$ such that
$$\delta_{s}\leq|\bs{v}^{s}|_{2}\leq\max\left\{1,\frac{s}{2}\right\}\delta_{s}$$
for $s=1,\dotsc,d$.
\end{theorem}

\noindent Theorem \ref{theorem:Mahler-Weyl} is a consequence of \cite[Chapter VIII, Lemma 1]{Cas97} and \cite[Chapter V, Lemma 8]{Cas97}.

For $d\geq 1$ and $\bs v\in\mb R^d$ let
$$\mathfrak{s}\left(\bs{v}\right):=\left\{h\in\{1,\dotsc,d\}:v_{h}\neq 0\right\}.$$
This notation will be used throughout the rest of the paper. From Theorem \ref{theorem:Mahler-Weyl}, we deduce the following.

\begin{lem}
\label{lem:monotonicity}
Let $\Lambda$ be a full rank lattice in $\mb{R}^{d}$ and let $\delta_{1},\dotsc,\delta_{d}$ be its successive minima. Then there exists a basis $\bs{v}^{1},\dotsc,\bs{v}^{d}$ of the lattice $\Lambda$ such that\vspace{2mm}
\begin{itemize}
\item[$i)$] $|\bs{v}^{s}|_{2}\ll_{d}\delta_{s}$ for $s=1,\dotsc,d$;\vspace{2mm}
\item[$ii)$] $\mathfrak{s}\left(\bs{v}^{s}\right)\subseteq \mathfrak{s}\left(\bs{v}^{s+1}\right)$ for $s=1,\dotsc,d-1$.\vspace{2mm}
\end{itemize}
\end{lem}

\begin{proof}
By Theorem \ref{theorem:Mahler-Weyl}, there exists a basis $\bs{v}^{1},\dotsc,\bs{v}^{d}$ of $\Lambda$, such that $|\bs{v}^{s}|_{2}\ll_{d}\delta_{s}$ for $s=1,\dotsc,d$. We set
\begin{equation}
\tilde{\bs{v}}^{s}:=\begin{cases}
\bs{v}^{1} & \quad\text{if }s=1 \\
\bs{v}^{s}+c_{s}\tilde{\bs{v}}^{s-1} & \quad\text{if }s>1
\end{cases},\nonumber
\end{equation}
where $c_{s}\in\mb{Z}$ are some coefficients yet to be chosen. Clearly, the vectors $\tilde{\bs{v}}^{1},\dotsc,\tilde{\bs{v}}^{d}$ form a basis of the lattice $\Lambda$ independently of the choice that we make for the coefficients $c_{s}$. We define these coefficients by recursion on $s$. Suppose that $\tilde{\bs{v}}^{s}$ are defined for all $s<\sigma $, where $1<\sigma \leq d$. To define $\tilde{\bs{v}}^{\sigma }$, we use the following procedure. First, we set $c_{\sigma }:=0$. If $\mathfrak{s}\left(\tilde{\bs{v}}^{\sigma -1}\right)\subset \mathfrak{s}\left(\bs{v}^{\sigma }\right)$, there is nothing to prove. If $\mathfrak{s}\left(\tilde{\bs{v}}^{\sigma -1}\right)\not\subset \mathfrak{s}\left(\bs{v}^{\sigma }\right)$, we change the value of the coefficient $c_{\sigma }$ to $1$. Then, if $\mathfrak{s}\left(\tilde{\bs{v}}^{\sigma -1}\right)\subset \mathfrak{s}\left(\bs{v}^{\sigma }+\tilde{\bs{v}}^{\sigma -1}\right)$, we have found a suitable value for the coefficient $c_{\sigma }$, otherwise, it means that the vectors $\tilde{\bs{v}}^{\sigma -1}$ and $\bs{v}^{\sigma }$ have some non-zero component of equal modulus but opposite sign. If this happens, we set $c_{\sigma }:=2$. Then either we have found a suitable value for the coefficient $c_{\sigma }$, or $\mathfrak{s}\left(\tilde{\bs{v}}^{\sigma -1}\right)\not\subset \mathfrak{s}\left(\bs{v}^{\sigma }+2\tilde{\bs{v}}^{\sigma -1}\right)$. This implies that some of the non-zero components of the vector $\bs{v}^{\sigma }$ are twice the same components of the vector $\tilde{\bs{v}}^{\sigma -1}$, up to a change of sign. If this is the case, we set $c_{\sigma }:=3$, and so on. As soon as $\mathfrak{s}\left(\tilde{\bs{v}}^{\sigma -1}\right)\subset \mathfrak{s}\left(\bs{v}^{\sigma }+c_{\sigma }\tilde{\bs{v}}^{\sigma -1}\right)$, we fix the value of the coefficient $c_{\sigma }$. Each non-zero component of the vector $\bs{v}^{\sigma }$ can exclude at most one value for the coefficient $c_{\sigma }$. Hence, the process terminates in at most $d+1$ steps.

To conclude, we show by recursion on $s$ that $\left|\bs{v}^{s}\right|\ll_{d}\delta_{s}$ for $s=1,\dotsc,d$. Suppose that this is true for all the indices less than a fixed index $\sigma >1$. Then, we have that
$$\left|\tilde{\bs{v}}^{\sigma }\right|_{2}=\left|\bs{v}^{\sigma -1}+c_{\sigma }\tilde{\bs{v}}^{\sigma -1}\right|_{2}\leq\left|\bs{v}^{\sigma }\right|_{2}+(d+1)\left|\tilde{\bs{v}}^{\sigma -1}\right|_{2}\ll_{d}\delta_{\sigma }+(d+1)\delta_{\sigma -1}\ll_{d}\delta_{\sigma },$$
and this completes the proof. 
\end{proof}

By Lemma \ref{lem:monotonicity}, there exists a basis $\bs{v}^{1},\dotsc,\bs{v}^{{m}+{n}}$ of the lattice $\omega_{2}\circ\omega_{1}\circ\tilde{\varphi}_{\beta}(\Lambda_{\bs{L}})$ such that
\begin{equation}
\label{eq:parti}
|\bs{v}^{s}|_{2}\ll_{{m}+{n}}\delta_{s}\mbox{ for }s=1,\dotsc,{m}+n,
\end{equation}
and
\begin{equation}
\label{eq:partii}
\mathfrak{s}\left(\bs{v}^{s}\right)\subseteq \mathfrak{s}\left(\bs{v}^{s+1}\right)\mbox{ for }s=1,\dotsc,{m}+{n}-1.
\end{equation}
Moreover, by definition of the maps $\omega_1,\omega_2,$ and $\varphi_{\beta}$ (see Proposition \ref{thm:partition}), we can write
\begin{equation}
\label{eq:partiii}
\bs{v}^{s}=\left(\theta e^{a_{\beta,1}}\left(L_{1}\bs{q}^{s}+p^{s}_{1}\right),\dotsc,\theta e^{a_{\beta,m}}\left(L_{{m}}\bs{q}^{s}+p^{s}_{{m}}\right),\theta^{-\frac{{m}}{{n}}}\frac{T}{T_{1}}q^{s}_{1},\dotsc,\theta^{-\frac{{m}}{{n}}}\frac{T}{T_{{n}}}q^{s}_{{n}}\right)
\end{equation}
for some fixed $\bs{p}^{s}\in\mb{Z}^{{m}}$ and $\bs{q}^{s}\in\mb{Z}^{{n}}$. Then, from (\ref{eq:parti}) and (\ref{eq:partiii}) we deduce that
\begin{multline}
\label{eq:length}
\delta_{s}\gg_{{m},{n}}|\bs{v}^{s}|_{2}=\Bigg(\theta^{2} e^{2a_{\beta,1}}\left(L_{1}\bs{q}^{s}+p^{s}_{1}\right)^{2}+\dotsb+\theta^{2} e^{2a_{\beta,m}}\left(L_{{m}}\bs{q}^{s}+p^{s}_{{m}}\right)^{2}+ \\
\left. +\ \theta^{-\frac{2{m}}{{n}}}\frac{T^{2}}{T_{1}^{2}}\left(q^{s}_{1}\right)^{2}+\dotsb+\theta^{-\frac{2{m}}{{n}}}\frac{T^{2}}{T^{2}_{{n}}}\left(q^{s}_{{n}}\right)^{2}\right)^{\frac{1}{2}}
\end{multline}
for $s=1,\dotsc,{m}+{n}$.

The following lemma shows that, without loss of generality, we can additionally assume that $\bs{q}^{s}\neq\bs{0}$ for all $s=1,\dotsc,{m}+{n}$.

\begin{lem}
\label{lem:qneq0}
Suppose that $\bs{q}^{\sigma_{0} }=\bs{0}$ for some $1\leq \sigma_{0} \leq {m}+{n}$. Then we have that
\begin{equation}
\frac{\left(\varepsilon T^{{n}}\right)^{\frac{s}{{m}+{n}}}}{\delta_{1}\dotsm\delta_{s}}\ll_{{m},{n}}R^{s}\nonumber
\end{equation}
for $s=1,\dotsc,{m}+{n}$.
\end{lem}

\begin{proof}
By (\ref{eq:partii}), we have that $\bs{q}^{s}=\bs{0}$ for all $s\leq \sigma_{0} $. It follows that $\bs{p}^{s}\neq\bs{0}$ for all $s\leq \sigma_{0} $. Hence, by part $iia)$ of Proposition \ref{thm:partition}, we deduce that
$$\delta_{s}\geq\delta_{1}\geq\theta\min_{i}e^{a_{\beta,i}}\gg_{{m},{n}}\frac{\left(\varepsilon T^{{n}}\right)^{\frac{1}{{m}+{n}}}}{\varepsilon^{\frac{1}{{m}}}}\frac{\varepsilon^{\frac{1}{{m}}}}{R}=\frac{1}{R}\left(\varepsilon T^{{n}}\right)^{\frac{1}{{m}+{n}}}$$
for all $s=1,\dotsc,{m}+{n}$. The claim follows directly from this inequality.
\end{proof}

Now, given that $\bs{q}^{s}\neq\bs{0}$ for all $s=1,\dotsc,{m}+{n}$, we deduce that $v^{s}_{i}=\theta e^{a_{\beta,i}}(L_{i}\bs{q}+p_{i})\neq 0$ for all $i=1,\dotsc,{m}$, since the entries of the row $L_{i}$ of the matrix $\bs{L}$ along with the integer $1$ are linearly independent over $\mb{Z}$ for $i=1,\dotsc,{m}$. Hence, in view of Lemma \ref{lem:qneq0}, we can make the assumption that
\begin{equation}
\label{eq:partiv}
\{1,\dotsc,{m}\}\subsetneq\mathfrak{s}\left(\bs{v}^{s}\right)
\end{equation}
for all $s=1,\dotsc,{m}+{n}$.

To conclude this subsection, we show that, by conveniently permuting the variables $y_{1},\dotsc ,y_{{n}}$, we can further assume that the sets $\mathfrak{s}(\bs{v}^{s})$ have a nice "triangular" structure. Here and in the sequel, we denote vectors in $\mb R^{m+n}$ by $(\bs x,\bs y)$, where $\bs x\in\mb R^{m}$ and $\bs y\in\mb R^{n}$.

\begin{lem}
\label{lem:monotonicity2}
There exists a permutation of the variables $y_{1},\dotsc,y_{{n}}$ such that for all the indices $s,l\in\{1,\dotsc,{m}+{n}\}$ we have that
\begin{equation}
\label{eq:partv}
l\in\mathfrak{s}\left(\bs{v}^{s}\right)\Rightarrow\{1,\dotsc,l\}\subset\mathfrak{s}\left(\bs{v}^{s}\right).
\end{equation}
\end{lem}

\begin{proof}
By (\ref{eq:partiv}) we can assume that
$$\{1,\dotsb,{m}\}\subsetneq \mathfrak{s}\left(\bs{v}^{s}\right)$$
for all $s=1,\dotsc,{m}+{n}$. Let us consider the sets
$$\mf{r}_{s}:=\mathfrak{s}\left(\bs{v}^{s}\right)\setminus\left\{1,\dotsc,{m}\right\}\neq\emptyset$$
for $s=1,\dotsc,{m}+{n}$, and let $\mf{r}_{0}:=\emptyset$. Let also $r_{s}:=\# R_{s}$ for $s=0,\dotsc,{m}+{n}$. To define the required permutation, we send the variables in the set $\{y_{j}:{m}+j\in \mf{r}_{s}\setminus \mf{r}_{s-1}\}$ to the variables in the set $\{y_{r_{s-1}+1},\dotsc,y_{r_{s}}\}$ for $s=1,\dotsc,{m}+{n}$, i.e., we reorder the variables so that the null components of each basis vector are the ones with higher indices. This simple procedure delivers the required result.
\end{proof}

Let ${m}+1\leq \sigma\leq{m}+{n}$. Since the vectors $\bs{v}^{1},\dotsc,\bs{v}^{\sigma}$ cannot all lie in the subspace $\{y_{\sigma-{m}}=\dotsb =y_{{n}}=0\}$ (which has dimension $\sigma-1$), there must be an index $s\leq \sigma$ such that one among the components $v^{s}_{\sigma},\dotsc,v^{s}_{{m}+{n}}$ of the vector $\bs{v}^{s}$ is non-zero. By (\ref{eq:partii}), we can take $s=\sigma$. Therefore, from (\ref{eq:partv}), we deduce that $v^{\sigma}_{{m}+1},\dotsc,v^{\sigma}_{\sigma}\neq 0$, and hence $\{1,\dotsc,\sigma\}\subset \mathfrak{s}\left(\bs{v}^{\sigma}\right)$. In view of this and of (\ref{eq:partiv}), we can assume that
\begin{equation}
\label{eq:partvi}
\{1,\dotsc,\max\{{m}+1,\sigma\}\}\subset\mathfrak{s}\left(\bs{v}^{\sigma}\right)
\end{equation}
for all $\sigma=1,\dotsc,{m}+{n}$.
Now, we have a sufficiently "nice" basis of the lattice $\omega_{2}\circ\omega_{1}\circ\tilde{\varphi}_{\beta}(\Lambda_{\bs{L}})$ and we may proceed to prove Proposition \ref{prop:firstminima}.

\subsection{Proof}

Throughout this section we fix and index $\beta\in{I}$ and a basis $\{\bs{v}^{1},\dotsc,\bs{v}^{{m}+{n}}\}$ of the lattice $\omega_{2}\circ\omega_{1}\circ\tilde{\varphi}_{\beta}(\Lambda_{\bs{L}})$ satisfying (\ref{eq:parti}), (\ref{eq:partii}), (\ref{eq:partiii}), (\ref{eq:partv}), and (\ref{eq:partvi}). Note that (\ref{eq:partvi}) implies that $q^{s}_{1}\neq 0$ for all $s=1,\dotsc,{m}+{n}$.

As mentioned in the Introduction, the main obstacle is represented by the null components of the basis vectors $\bs{v}^{1},\dotsc,\bs{v}^{{m}+{n}}$. We start by presenting an instructive "naive approach" to the problem, that shows how estimates for the first minimum only are not sufficient to conclude.  

\begin{lem}
\label{lem:sleqM}
Let $1\leq h_{1}\leq{n}$ be the largest index such that $q^{1}_{h_{1}}\neq 0$. Then, for all $\sigma =1,\dotsc,{m}+{n}$ we have that
\begin{equation}
\label{eq:sleqM}
\frac{\left(\varepsilon T^{{n}}\right)^{\frac{\sigma }{{m}+{n}}}}{\delta_{1}\dotsm\delta_{\sigma }}\ll_{{m},{n}}1+\left(\frac{\varepsilon T^{{n}}}{\phi\left(T\right)}\right)^{\frac{\sigma }{{m}+h_{1}}}.
\end{equation}
\end{lem}

\begin{proof}
From (\ref{eq:length}), we have that
\begin{multline}
\label{eq:lengthlem1}
\delta_{1}\gg_{{m},{n}}\Bigg(\theta^{2} e^{2a_{\beta,1}}\left\|L_{1}\bs{q}^{1}\right\|^{2}+\dotsb+\theta^{2} e^{2a_{\beta,m}}\left\|L_{{m}}\bs{q}^{1}\right\|^{2}+ \\
\left. +\ \theta^{-\frac{2{m}}{{n}}}\frac{T^{2}}{T_{1}^{2}}\left(q^{1}_{1}\right)^{2}+\dotsb+\theta^{-\frac{2{m}}{{n}}}\frac{T^{2}}{T^{2}_{{n}}}\left(q^{1}_{h_{1}}\right)^{2}\right)^{\frac{1}{2}}.
\end{multline}
We consider two different cases. Let us first assume that for all $j=1,\dotsc,{n}$, it holds $\left|q^{1}_{j}\right|\leq T_{j}$.
Then, by applying the standard arithmetic-geometric mean inequality to the right-hand side of (\ref{eq:lengthlem1}), we find that
\begin{equation}
\label{eq:z}
\delta_{1}\gg_{{m},{n}}\left(\theta^{{m}\left(1-\frac{h_{1}}{{n}}\right)}\prod_{i=1}^{{m}}\left\|L_{i}\bs{q}^{1}\right\|\cdot
\prod_{j=1}^{h_{1}}\frac{T}{T_{j}}\left|q^{1}_{j}\right|\right)^{\frac{1}{{m}+h_{1}}}.
\end{equation}
Since the matrix $\bs{L}$ is multiplicatively badly approximable, we have that
\begin{equation}
\label{eq:zz}
\prod_{i=1}^{{m}}\left\|L_{i}\bs{q}^{1}\right\|\geq\frac{\phi\left(\left(q_{1}^{1+}\dotsm q_{n}^{1+}\right)^{\frac{1}{{n}}}\right)}{q_{1}^{1+}\dotsm q_{n}^{1+}}\geq\frac{\phi\left(T\right)}{\prod_{j=1}^{h_{1}}\left|q^{1}_{j}\right|}.
\end{equation}
Moreover,
\begin{equation}
\label{eq:zzz}
\prod_{j=1}^{h_{1}}\frac{T}{T_{j}}\geq T^{h_{1}-{n}}.
\end{equation}
Substituting (\ref{eq:zz}) and (\ref{eq:zzz}) into (\ref{eq:z}), we conclude that
\begin{equation*}
\delta_{1}\gg_{{m},{n}}\left(\theta^{{m}\left(1-\frac{h_{1}}{{n}}\right)}\phi(T)T^{h_{1}-{n}}\right)^{\frac{1}{{m}+h_{1}}}=\left(\varepsilon T^{{n}}\right)^{-\frac{1}{{m}+h_{1}}+\frac{1}{{m}+{n}}}\phi(T)^{\frac{1}{{m}+h_{1}}},
\end{equation*}
where, in the second equality, we used the definition of $\theta$. This implies that
\begin{equation}
\label{eq:lengthlem12}
\frac{\left(\varepsilon T^{{n}}\right)^{\frac{\sigma }{{m}+{n}}}}{\delta_{1}\dotsm\delta_{\sigma }}\leq\frac{\left(\varepsilon T^{{n}}\right)^{\frac{\sigma }{{m}+{n}}}}{\delta_{1}^{\sigma }}\ll_{{m},{n}}\left(\frac{\varepsilon T^{{n}}}{\phi\left(T\right)}\right)^{\frac{\sigma }{{m}+h_{1}}},
\end{equation}
completing the proof in this case.

Let us now assume that there exists an index $1\leq j_{0}\leq h_{1}$ such that $\left|q^{1}_{j_{0}}\right|> T_{j_{0}}$. By ignoring all the terms but $\theta^{-2{m}/{n}}\left(T/T_{j_{0}}\right)^{2}(q^{1}_{j_{0}})^{2}$ in (\ref{eq:lengthlem1}), we deduce that
\begin{equation*}
\delta_{1}\gg_{{m},{n}}\theta^{-\frac{{m}}{{n}}}T=\left(\varepsilon T^{{n}}\right)^{\frac{1}{{m}+{n}}}.
\end{equation*}
Hence,
\begin{equation}
\label{eq:lengthlem13}
\frac{\left(\varepsilon T^{{n}}\right)^{\frac{\sigma }{{m}+{n}}}}{\delta_{1}\dotsm\delta_{\sigma }}\leq\frac{\left(\varepsilon T^{{n}}\right)^{\frac{\sigma }{{m}+{n}}}}{\delta_{1}^{\sigma }}\ll_{{m},{n}}1.
\end{equation}
The claim follows from (\ref{eq:lengthlem12}) and (\ref{eq:lengthlem13}).
\end{proof}

\begin{rmk}
\label{rmk:whylem1doesntwork}
In order to prove Proposition \ref{prop:firstminima}, we need the exponent of the ratio $\varepsilon T^{{n}}/\phi(T)$ in (\ref{eq:sleqM}) to be less than or equal to $({m}+{n}-1)/({m}+{n})$. The fact that this exponent is strictly less than $1$ is crucial to obtain Theorem \ref{thm:BHV}. Now, Lemma \ref{lem:sleqM} ensures that this holds true for $\sigma \leq{m}$. However, for $\sigma \geq{m}+1$, the result depends on the value of the index $h_{1}$. In particular, to deduce the desired estimate for
$$\frac{\left(\varepsilon T^{{n}}\right)^{\frac{\sigma }{{m}+{n}}}}{\delta_{1}\dotsm\delta_{\sigma }},$$
we need the number ${m}+h_{1}$ to be at least $\sigma +1$, i.e., we need that $\{1,\dotsc,\sigma \}\subset\mathfrak{s}(\bs{v}^{1})$. This cannot be guaranteed, since the only information that we have with regards to the vector $\bs{q}^{1}$ is that $\bs{q}^{1}\neq\bs{0}$.
\end{rmk}

In view Remark \ref{rmk:whylem1doesntwork}, a slightly more sophisticated approach is required, where the higher successive minima $\delta_{2},\dotsc,\delta_{s}$ of the lattice $\omega_{2}\circ\omega_{1}\circ\tilde{\varphi}_{\beta}(\Lambda_{\bs{L}})$ play a crucial role.

We will show that for a fixed index $\sigma $ it is not necessary to have $\{1,\dotsc,\sigma +1\}\subset\mathfrak{s}(\bs{v}^{1})$ to obtain (\ref{eq:firstminima}), but it suffices that $\{1,\dotsc,s+1\}\subset\mathfrak{s}(\bs{v}^{s})$ for all $s=1,\dotsc,\sigma $. This is made precise in the following lemma.

\begin{lem}
\label{lem:quickcase}
Let ${m}\leq\sigma \leq {m}+{n}-1$ and suppose that for all $s=1,\dotsc,\sigma $ it holds $\{1,\dotsc,s+1\}\subseteq \mathfrak{s}\left(\bs{v}^{s}\right)$. Then, we have that
$$\frac{\left(\varepsilon T^{{n}}\right)^{\frac{\sigma }{{m}+{n}}}}{\delta_{1}\dotsm\delta_{\sigma }}\ll_{{m},{n}}1+\left(\frac{\varepsilon T^{{n}}}{\phi\left(T\right)}\right)^{\frac{\sigma }{\sigma +1}}.$$
\end{lem}
\noindent The proof of Lemma \ref{lem:quickcase} is rather involved and we postpone it to Section \ref{sec:quickcase}.

\begin{rmk}
Note that Lemma \ref{lem:quickcase} cannot be proved by adapting the strategy used to prove Lemma \ref{lem:sleqM} to higher successive minima. Specifically, the resulting bound for the exponent of $\varepsilon T^{n}/\phi(T)$ through this approach would be
\begin{equation}
\label{eq:notexponent}
\sum_{s=1}^{\sigma}(m+h_s)^{-1},    
\end{equation}
where $h_s$ is the largest non-zero index $j$ such that $q_j\neq 0$. However, assuming that $h_s=s+1$, one sees that \eqref{eq:notexponent} tends to infinity as $n-m$ grows.
\end{rmk}

Let us fix an index $\sigma \in\{1,\dotsc,{m}+{n}\}$. Recall that Lemma \ref{lem:quickcase} is applicable whenever the condition $\{1,\dotsc,s+1\}\subset\mathfrak{s}\left(\bs{v}^{s}\right)$ holds for all $s=1,\dotsc,\sigma $. However, the condition $\{1,\dotsc,s+1\}\subset\mathfrak{s}\left(\bs{v}^{s}\right)$, for example, never holds for $s={m}+{n}$. Thus, we are left with one extra case to consider, i.e., the case when there exists an index $s\leq\sigma $ such that $\mathfrak{s}\left(\bs{v}^{s}\right)=\{1,\dotsc,s\}$ (recall that \eqref{eq:partvi} holds). We deal with this case in the following lemma.

\begin{lem}
\label{lem:slowcase}
Let ${m}+1\leq\sigma \leq {m}+{n}$ and assume that there exists an index ${m}+1\leq \sigma_{0} \leq\sigma $ such that $\mathfrak{s}\left(\bs{v}^{\sigma_{0} }\right)=\{1,\dotsc,\sigma_{0} \}$. Then we have that
$$\frac{\left(\varepsilon T^{{n}}\right)^{\frac{\sigma }{{m}+{n}}}}{\delta_{1}\dotsm\delta_{\sigma }}\ll_{{m},{n}}\varepsilon T^{{n}}.$$
\end{lem}
\noindent The proof of Lemma \ref{lem:quickcase} can be found in Section \ref{sec:slowcase}.

Combining Lemmas \ref{lem:qneq0}, \ref{lem:sleqM}, \ref{lem:quickcase}, and \ref{lem:slowcase}, we finally obtain a proof of Proposition \ref{prop:firstminima}. 

\section{Counting Lattice Points}
\label{sec:proofM1}

This section is devoted to the proof of Lemmas \ref{lem:quickcase} and \ref{lem:slowcase}.

\subsection{Proof of Lemma \ref{lem:quickcase}}
\label{sec:quickcase}

We distinguish two cases. Let us first assume that for all $1\leq s\leq\sigma $ and all $1\leq j\leq{n}$ it holds $\left|q^{s}_{j}\right|\leq T_{j}$. Let $1\leq h_{s}\leq {n}$ be the largest index such that $q^{s}_{h_{s}}\neq 0$. By (\ref{eq:length}), we have that
\begin{align}
\label{eq:lengthlem3}
\left|\bs{v}^{s}\right|_{2} & \geq\Bigg(\theta^{2} e^{2a_{\beta,1}}\left\|L_{1}\bs{q}^{s}\right\|^{2}+\dotsb+\theta^{2} e^{2a_{\beta,m}}\left\|L_{{m}}\bs{q}^{s}\right\|^{2}+\nonumber \\
 & \hspace{5cm}+\theta^{-\frac{2{m}}{{n}}}\frac{T^{2}}{T_{1}^{2}}\left(q^{s}_{1}\right)^{2}+\dotsb+\frac{T^{2}}{T^{2}_{h_{s}}}\left(q^{s}_{h_{s}}\right)^{2}\Bigg)^{\frac{1}{2}},
\end{align} 
where $s$ ranges from $1$ to $\sigma $. Recall that, by (\ref{eq:partv}), the condition $q^{s}_{h_{s}}\neq 0$ implies that $q^{s}_{1},\dotsc,q^{s}_{h_{s}}\neq 0$ for all $s=1,\dotsc,\sigma $. Moreover, by (\ref{eq:partvi}) and the hypothesis, we have that $h_{s}\geq\max\{{m}+1,s+1\}$.

To bound below the right-hand side of (\ref{eq:lengthlem3}), for each $s=1,\dotsc,\sigma $ we use a weighted arithmetic-geometric mean inequality. Let us explain the logic behind the choice of the weights. If the condition $\{1,\dotsc,\sigma +1\}\subset\mathfrak{s}(\bs{v}^{s})$ were true for $s=1,\dotsc,\sigma $, a standard arithmetic-geometric mean inequality at each level $s$ would suffice (see Remark \ref{rmk:whylem1doesntwork}). 
Since it is not always true that $\{1,\dotsc,\sigma +1\}\subset\mathfrak{s}(\bs{v}^{s})$ for all $s=1,\dotsc,\sigma $,
when applying the weighted arithmetic-geometric mean inequality at the level $s$, we assign a heavier weight to the term $v^{s}_{s+1}$, to compensate for the potential absence of the $(s+1)$-th entry in the vectors $\bs{v}^{s'}$ for $s'<s$. By doing so, we essentially recover the condition that $\{1,\dotsc,\sigma +1\}\subset\mathfrak{s}(\bs{v}^{s})$ for all $s=1,\dotsc,\sigma $.

Let us define the weights more precisely. Each component $v_j^s$ of the vector $\bs v^s$ ($s=1,\dotsc,\sigma$) is assigned a weight $w_{sr}$ for $r=1,\dotsc,m+h_s$ (recall that $h_s$ is the largest index $j$ such that $v^s_{j}\neq 0$). We call weighted arithmetic-geometric mean inequality, with weights $w_{sr}$, the following relation\footnote{This follows from the inequality $w_1 x_1+\dotsb +w_dx_d\geq x_1^{w_1}\dotsm x_d^{w_d}$, valid for all positive $x_1,\dotsc,x_d,w_1,\dotsc,w_d\in\mb R$ with $\sum_{i}w_i=1$. A proof of this may be obtained by applying the finite form of Jensen's Inequality to the logarithm function.}:
\begin{equation}
\label{eq:WAM-GM}
\left(\left(v_{1}^{s}\right)^{2}+\dotsb+\left(v_{m+h_s}^{s}\right)^{2}\right)^{1/2}\geq\left(v_{1}^{s}\right)^{w_{s1}}\dotsm \left(v_{m+h_s}^{s}\right)^{w_{s(m+h_s)}}.
\end{equation}
For $r=1,\dotsc,m+h_s$ we define
\begin{equation*}
w_{sr}=\begin{cases}
\dfrac{1}{m+h_s} & \mbox{if }s\leq m  \\[3mm]  
\dfrac{1}{k_s} & \mbox{if }s>m\mbox{ and }r\neq s+1  \\[3mm] 
\dfrac{1+(k_{s}-{m}-h_{s})}{k_{s}} & \mbox{if }s>m\mbox{ and }r=s+1
\end{cases},    
\end{equation*}
where $k_{s}$ is a parameter yet to be defined. Note that for each $s$ the weights sum up to $1$. Now, let us proceed to define $k_s$. For $s=1,\dotsc,{m}$, according to our choice of weights, we trivially put 
\begin{equation}
\label{eq:exponentbasecase}
k_{s}:={m}+h_{s}.
\end{equation}
For $s\geq {m}+1$, we define $k_{s}$ by recursion. Let $\bs{t}$ be the $({m}+{n}-1)\times {n}$ matrix whose entries are defined by
$$t_{sj}:=
\begin{cases}
1 & \quad\text{if }q^{s}_{j}\neq 0 \\
0 & \quad\text{if }q^{s}_{j}=0
\end{cases}.$$
Then, in line with what we explained above, we require that the numbers $k_{s}$ satisfy the equation
\begin{multline}
\label{eq:fillingup}
\frac{1-t_{1(s+1-{m})}}{k_{1}}+\dotsb+\frac{1-t_{(s-1)(s+1-{m})}}{k_{s-1}}+\frac{1}{k_{s}} \\
=\frac{1+(k_{s}-{m}-h_{s})}{k_{s}}=w_{s(s+1-m)}
\end{multline}
for $s={m}+1,\dotsc,\sigma $. Note that the left-hand side counts (with weights) how many times the $(s+1)$-th entry is null in the vectors $\bs{v}^{1},\dotsc,\bs{v}^{s}$ and adds to this $1/k_s$, representing the component $v_{s+1}^{s}$ itself. Imposing (\ref{eq:fillingup}) for a specific value of the index $s$ allows us to balance out the absence of the terms $T/T_{s+1-{m}}$ in the product $\left|\bs{v}^{1}\right|_{2}\dotsm\left|\bs{v}^{\sigma }\right|_{2}$. This is crucial to obtain an upper bound depending only on the value of $T$ and not on $\max_{j}T_{j}$ in Proposition \ref{prop:firstminima}.

From (\ref{eq:fillingup}) we deduce that
\begin{equation}
\label{eq:from40}
k_{s}:=({m}+h_{s})\left(1-\left(\frac{1-t_{1(s+1-{m})}}{k_{1}}+\dotsb+\frac{1-t_{(s-1)(s+1-{m})}}{k_{s-1}}\right)\right)^{-1}
\end{equation}
for $s={m}+1,\dotsc,\sigma $. Therefore, in order to show that the numbers $k_{s}$ are well defined, we have to prove that
\begin{equation}
1-\left(\frac{1-t_{1(s+1-{m})}}{k_{1}}+\dotsb+\frac{1-t_{(s-1)(s+1-{m})}}{k_{s-1}}\right)>0\nonumber
\end{equation}
for $s={m}+1,\dotsc,\sigma $. This always holds true, thanks to the following (more precise) result.
\begin{lem}
\label{lem:exponentswelldef}
Let $k_{s}>0$ be real numbers such that (\ref{eq:exponentbasecase}) and (\ref{eq:fillingup}) hold for $s=1,\dotsc,{m}+{n}-1$. Let also
\begin{equation}
\label{eq:alphasdef}
\alpha_{s}:=\frac{1}{k_{1}}+\dotsb+\frac{1}{k_{s}}    
\end{equation}
and
\begin{equation}
\label{eq:alphasjdef}
\alpha_{sj}:=\frac{t_{1j}}{k_{1}}+\dotsb+\frac{t_{sj}}{k_{s}}    
\end{equation}
for $s=1,\dotsc,{m}+{n}-1$ and $j=1,\dots,{n}$. Then, under the hypotheses of Lemma \ref{lem:quickcase} and provided (\ref{eq:partv}) and (\ref{eq:partvi}) hold, we have that\vspace{2mm}
\begin{itemize}
\item[$i)$] $k_{s}\geq{m}+h_{s}$ for $s=1,\dotsc,{m}+{n}-1$;\vspace{2mm}
\item[$ii)$] $\alpha_{s}(s+1)+\sum_{j>s+1-{m}}\alpha_{sj}=s$ for $s={m},\dotsc,{m}+{n}-1$, whence $\alpha_{s}\leq s/(s+1)$.\vspace{2mm}
\end{itemize}
If $s+1-{m}={n}$, the sum $\sum_{j>s+1-{m}}\alpha_{sj}$ in part $ii)$ should be disregarded.
\end{lem}
\noindent We prove Lemma \ref{lem:exponentswelldef} in Section \ref{sec:exponentwelldef}.

Let us now apply (\ref{eq:WAM-GM}). In view of Proposition \ref{thm:partition} part $iia)$ and by definition of the weights $w_{sr}$, we obtain
\begin{equation}
\left|\bs{v}^{s}\right|_{2}\gg_{{m},{n}}\left(\theta^{{m}\left(1-\frac{k_{s}-{m}}{{n}}\right)}\prod_{i=1}^{{m}}\left\|L_{i}\left(\bs{q}^{s}\right)\right\|\cdot
\prod_{j=1}^{h_{s}}\frac{T}{T_{j}}\left|q^{s}_{j}\right|\cdot\left(\frac{T}{T_{s+1-{m}}}\left|q^{s}_{s+1-{m}}\right|\right)^{k_{s}-{m}-h_{s}}\right)^{\frac{1}{k_{s}}}\nonumber
\end{equation}
for $s=1,\dotsc,\sigma $. Since the matrix $\bs{L}$ is multiplicatively badly approximable, we also have that
$$\prod_{i=1}^{{m}}\left\|L_{i}\left(\bs{q}^{s}\right)\right\|\geq\frac{\phi\left(\left(q^{s+}_{1}\dotsm q_{n}^{s+}\right)^{\frac{1}{{n}}}\right)}{q^{s+}_{1}\dotsm q_{n}^{s+}}\geq\frac{\phi\left(T\right)}{\prod_{j=1}^{h_{s}}\left|q^{s}_{j}\right|}.$$
for all $s=1,\dots,\sigma $. Hence, we deduce that
\begin{equation}
\label{eq:lengthlem4'}
\left|\bs{v}^{s}\right|_{2}\gg_{{m},{n}}\left(\theta^{{m}\left(1-\frac{k_{s}-{m}}{{n}}\right)}\cdot\phi\left(T\right)\cdot\prod_{j=1}^{h_{s}}\frac{T}{T_{j}}\cdot\left(\frac{T}{T_{s+1-{m}}}\right)^{k_{s}-{m}-h_{s}}\right)^{\frac{1}{k_{s}}}
\end{equation}
for $s=1,\dotsc,\sigma $. We are then left to estimate the product $\left|\bs{v}^{1}\right|_{2}\dotsm\left|\bs{v}^{\sigma }\right|_{2}$. We do this by multiplying together the inequalities in (\ref{eq:lengthlem4'}). For simplicity, we compute the exponent of each factor ($\theta$, $\phi(T)$, and $T/T_j$) separately, using the notation introduced in Lemma \ref{lem:exponentswelldef}, namely \eqref{eq:alphasdef} and \eqref{eq:alphasjdef}. For the constant $\theta$ the exponent in the product of the inequalities in (\ref{eq:lengthlem4'}) is given by
\begin{equation}
\frac{{m}}{k_{1}}\left(1-\frac{k_{1}-{m}}{{n}}\right)+\dotsb+\frac{{m}}{k_{\sigma }}\left(1-\frac{k_{\sigma }-{m}}{{n}}\right)=\frac{{m}({m}+{n})}{{n}}\alpha_{\sigma }-\frac{{m}}{{n}}\sigma,\nonumber
\end{equation}
while for the function $\phi\left(T\right)$ the same exponent equals $\alpha_{\sigma }$. The case of the ratios $T/T_j$ is less evident. If $1\leq j\leq\sigma +1-{m}$ the exponent in the product equals
\begin{equation*}
\sum_{s<j-1+m}\frac{t_{sj}}{k_s}+\frac{k_{j-1+m}-m-h_{j-1+m}}{k_{j-1+m}}+\sum_{j-1+m<s\leq \sigma}\frac{1}{k_s},    
\end{equation*}
where the exceptional value $j-1+m$ is obtained by equating $j=s+1-m$. By applying \eqref{eq:fillingup}, we deduce that this quantity is also equal to
\begin{equation}
\label{eq:tjexp}
\sum_{s<j-1+m}\frac{t_{sj}}{k_s}+\sum_{s\leq j-1+m}\frac{1}{k_s}-\sum_{s<j-1+m}\frac{t_{sj}}{k_s}+\sum_{j-1+m<s\leq \sigma}\frac{1}{k_s}=\alpha_{\sigma}.
\end{equation}
On the other hand, if $j\geq\sigma +2-{m}$, the exponent of the ratios $T/T_{j}$ is given by $\alpha_{\sigma j}$. Combining these considerations, we obtain that
\begin{equation}
\label{eq:lengthlem5}
\left|\bs{v}^{1}\right|_{2}\dotsm\left|\bs{v}^{\sigma }\right|_{2}\gg_{{m},{n}} \theta^{\frac{{m}({m}+{n})}{{n}}\alpha_{\sigma }-\frac{{m}}{{n}}\sigma }\cdot\phi\left(T\right)^{\alpha_{\sigma }}\cdot
\left(\prod_{j=1}^{\sigma +1-{m}}\frac{T}{T_{j}}\right)^{\alpha_{\sigma }}\prod_{j=\sigma +2-{m}}^{{n}}\left(\frac{T}{T_{j}}\right)^{\alpha_{\sigma j}}.
\end{equation}
Now, we observe that
\begin{equation}
\label{eq:finaltheta}
\theta^{\frac{{m}({m}+{n})}{{n}}\alpha_{\sigma }-\frac{{m}}{{n}}\sigma }=\left(\varepsilon^{-\frac{{n}}{{m}({m}+{n})}}T^{\frac{{n}}{{m}+{n}}}\right)^{\frac{{m}({m}+{n})}{{n}}\alpha_{\sigma }-\frac{{m}}{{n}}\sigma }=\varepsilon^{-\alpha_{\sigma }+\frac{\sigma }{{m}+{n}}}T^{{m}\alpha_{\sigma }-\frac{{m}}{{m}+{n}}\sigma }.
\end{equation}
Moreover, by using the fact that $T$ is the geometric mean of the parameters $T_{1},\dotsc,T_{{n}}$, we have that
\begin{multline}
\label{eq:expratios3}
\left(\prod_{j=1}^{\sigma +1-{m}}\frac{T}{T_{j}}\right)^{\alpha_{\sigma }}\prod_{j=\sigma +2-{m}}^{{n}}\left(\frac{T}{T_{j}}\right)^{\alpha_{\sigma j}}= \\
=\prod_{j=\sigma +2-{m}}^{{n}}\left(\frac{T}{T_{j}}\right)^{\alpha_{\sigma j}-\alpha_{\sigma }}\geq T^{-\alpha_{\sigma }({n}+{m}-\sigma -1)+\sum_{j=\sigma +2-{m}}^{{n}}\alpha_{\sigma j}},
\end{multline}
where the lower bound is obtained by trivially setting $T_{j}=1$ for $\sigma +2-{m}\leq j\leq {n}$. By Lemma \ref{lem:exponentswelldef}, it holds
$$\alpha_{\sigma }(\sigma +1)+\sum_{j=\sigma +2-{m}}^{{n}}\alpha_{\sigma j}=\sigma ,$$
therefore, (\ref{eq:expratios3}) implies that
\begin{equation}
\label{eq:nnew}
\left(\prod_{j=1}^{\sigma +1-{m}}\frac{T}{T_{j}}\right)^{\alpha_{\sigma }}\prod_{j=\sigma +2-{m}}^{{n}}\left(\frac{T}{T_{j}}\right)^{\alpha_{\sigma j}}\geq T^{-\alpha_{\sigma }({n}+{m})+\sigma }.
\end{equation}
Combining (\ref{eq:finaltheta}) and (\ref{eq:nnew}) with (\ref{eq:lengthlem5}), we finally obtain that
\begin{multline*}
\left|\bs{v}^{1}\right|_{2}\dotsm\left|\bs{v}^{\sigma }\right|_{2}\gg_{{m},{n}}\varepsilon^{-\alpha_{\sigma }+\frac{\sigma }{{m}+{n}}}T^{{m}\alpha_{\sigma }-\frac{{m}}{{m}+{n}}\sigma -\alpha_{\sigma }({n}+{m})+\sigma }\phi(T)^{\alpha_{\sigma }} \\
=\left(\varepsilon T^{{n}}\right)^{-\alpha_{\sigma }+\frac{\sigma }{{m}+{n}}}\phi(T)^{\alpha_{\sigma }}.    
\end{multline*}

This, in turn, yields
$$\frac{\left(\varepsilon T^{{n}}\right)^{\frac{\sigma }{{m}+{n}}}}{\delta_{1}\dotsm\delta_{\sigma }}\ll_{{m},{n}}\left(\frac{\varepsilon T^{{n}}}{\phi\left(T\right)}\right)^{\alpha_{\sigma }}\ll_{{m},{n}}1+\left(\frac{\varepsilon T^{{n}}}{\phi\left(T\right)}\right)^{\frac{\sigma }{\sigma +1}},$$
where the last inequality is due to Lemma \ref{lem:exponentswelldef} part $ii)$. This completes the proof in the case $|q_{j}^s|\leq T_j$.

Let us now assume that there exist indices $1\leq \sigma_{0} \leq\sigma $ and $1\leq j_{0}\leq {n}$ such that $\left|q^{\sigma_{0} }_{j_{0}}\right|> T_{j_{0}}$. Suppose that $\sigma_{0} $ is the least index such that $\left|q^{\sigma_{0} }_{j_{0}}\right|> T_{j_{0}}$ for some $1\leq j_{0}\leq {n}$. Then, by ignoring all the terms but $\theta^{-2{m}/{n}}\left(T/T_{j_{0}}\right)^{2}(q^{\sigma_{0} }_{j_{0}})^{2}$ in (\ref{eq:lengthlem3}), we find that
$$\delta_{s}\gg_{{m},{n}}\theta^{-\frac{{m}}{{n}}}T=\varepsilon^{\frac{1}{{m}+{n}}}T^{\frac{{n}}{{m}+{n}}}$$
for all $\sigma_{0} \leq s\leq\sigma $. Hence, we can write
$$\frac{\left(\varepsilon T^{{n}}\right)^{\frac{\sigma }{{m}+{n}}}}{\delta_{1}\dotsm\delta_{\sigma }}\ll_{{m},{n}}\frac{\left(\varepsilon T^{{n}}\right)^{\frac{\sigma_{0} -1}{{m}+{n}}}}{\delta_{1}\dotsm\delta_{\sigma_{0} -1}}\frac{\left(\varepsilon T^{{n}}\right)^{\frac{\sigma -\sigma_{0} +1}{{m}+{n}}}}{\left(\varepsilon T^{{n}}\right)^{\frac{\sigma -\sigma_{0} +1}{{m}+{n}}}}=\frac{\left(\varepsilon T^{{n}}\right)^{\frac{\sigma_{0} -1}{{m}+{n}}}}{\delta_{1}\dotsm\delta_{\sigma_{0} -1}}.$$
If $\sigma_{0} \geq{m}+1$, one concludes by the case $|q_{j}^s|\leq T_j$. Otherwise, Lemma \ref{lem:sleqM} gives the required estimate.

\subsection{Proof of Lemma \ref{lem:slowcase}}
\label{sec:slowcase}

Let $B_{\beta}\in\mb{R}^{({m}+{n})\times({m}+{n})}$ be the matrix that represents the linear transformation $\omega_{2}\circ\omega_{1}\circ\tilde{\varphi}_{\beta}$ in the canonical basis. Then we have that
$$\omega_{1}\circ\omega_{2}\circ\tilde{\varphi}_{\beta}(\Lambda_{\bs{L}})=B_{\beta}A_{\bs{L}}\mb{Z}^{{m}+{n}},$$
where $A_{\bs L}$ is defined in \eqref{eq:AL}. Let $\left(B_{\beta}A_{\bs{L}}\right)_{\sigma_{0} }$ be the $\sigma_{0} \times \sigma_{0} $ submatrix of $B_{\beta}A_{\bs{L}}$ formed by the first $\sigma_{0} $ rows and the first $\sigma_{0} $ columns. Let also $\Lambda_{\sigma_0}:=\left(B_{\beta}A_{\bs{L}}\right)_{\sigma_{0} }\mb Z^{\sigma_0}$. By the definition of the maps $\omega_1,\omega_2$, and Proposition \ref{thm:partition}, the lattice $\Lambda_{\sigma_{0} }$ has rank $\sigma_{0} $ and co-volume
\begin{equation}
\label{eq:slowcase0}
\Det\left(\Lambda_{\sigma_{0} }\right)=\theta^{{m}\left(1-\frac{\sigma_{0} -{m}}{{n}}\right)}\prod_{j=1}^{\sigma_{0} -{m}}\frac{T}{T_{j}}.
\end{equation}
Moreover, by the hypothesis, we have that $\bs{v}_{1},\dotsc,\bs{v}_{\sigma_{0} }\in\Lambda_{\sigma_{0} }\times\{\bs 0\}.$
Since $\bs{v}_{1},\dotsc,\bs{v}_{\sigma_{0} }$ are linearly independent, by Minkowsky's Theorem (see \cite[Chapter VIII, Theorem I]{Cas97}), we deduce that
\begin{equation}
\label{eq:slowcase1}
\delta_{1}\dotsm\delta_{\sigma_{0} }\gg_{{m},{n}}\prod_{s=1}^{\sigma_{0} }\left|\bs{v}^{s}\right|_{2}\geq\delta_{1}(\Lambda_{\sigma_{0} })\dotsm\delta_{\sigma_{0} }(\Lambda_{\sigma_{0} })\gg_{{m},{n}}\Det\left(\Lambda_{\sigma_{0} }\right),
\end{equation}
where $\delta_{s}(\Lambda_{\sigma_{0} })$ ($s=1,\dotsc,\sigma_{0} $) are the successive minima of the lattice $\Lambda_{\sigma_{0} }$. This gives us an estimate of the product $\delta_{1}\dotsm\delta_{\sigma_{0} }$.

We are now left to estimate the product $\delta_{\sigma_{0} +1}\dotsm\delta_{\sigma }$. By (\ref{eq:partvi}), we have that $\{1,\dotsc,s\}\subseteq \mathfrak{s}\left(\bs{v}^{s}\right)$ for $s={m}+1,\dotsc,{m}+{n}$. Hence, by (\ref{eq:length}), where we ignore all the terms but $\theta^{-\frac{{m}}{{n}}}\frac{T}{T_{s-{m}}}\left|q^{s}_{s-{m}}\right|$, we obtain that
\begin{equation}
\label{eq:slowcase3}
\delta_{s}\gg_{{m},{n}}|\bs{v}^{s}|_{2}\geq\theta^{-\frac{{m}}{{n}}}\frac{T}{T_{s-{m}}}\left|q^{s}_{s-{m}}\right|\geq\theta^{-\frac{{m}}{{n}}}\frac{T}{T_{s-{m}}}
\end{equation}
for $s=\sigma_{0} +1,\dotsc,\sigma$. Then (\ref{eq:slowcase0}), (\ref{eq:slowcase1}), and (\ref{eq:slowcase3}) imply that
\begin{align}
\delta_{1}\dotsm\delta_{\sigma } & \gg_{{m},{n}}\Det\left(\Lambda_{\sigma_{0} }\right)\prod_{s=\sigma_{0} +1}^{\sigma }\theta^{-\frac{{m}}{{n}}}\frac{T}{T_{s-{m}}}=\nonumber \\
 & =\theta^{{m}\left(1-\frac{\sigma_{0} -{m}}{{n}}\right)}\prod_{j=1}^{\sigma_{0} -{m}}\frac{T}{T_{j}}\prod_{j=\sigma_{0} -{m}+1}^{\sigma -{m}}\theta^{-\frac{{m}}{{n}}}\frac{T}{T_{j}}=\nonumber \\
 & =\theta^{{m}\left(1-\frac{\sigma -{m}}{{n}}\right)}\frac{1}{T^{{m}+{n}-\sigma }}\prod_{j=\sigma +1-{m}}^{{n}}T_{j}\geq\nonumber \\
 & \geq\theta^{{m}\left(1-\frac{\sigma -{m}}{{n}}\right)}T^{\sigma -({m}+{n})}=\left(\varepsilon T^{{n}}\right)^{\frac{\sigma }{{m}+{n}}-1},\nonumber
\end{align}
where $\prod_{s=\sigma_{0} +1}^{\sigma }\theta^{-\frac{{m}}{{n}}}T/T_{s-{m}}=1$ if $\sigma_{0} =\sigma $.
This gives
\begin{equation}
\frac{\left(\varepsilon T^{{n}}\right)^{\frac{\sigma }{{m}+{n}}}}{\delta_{1}\dotsm\delta_{\sigma }}\ll_{{m},{n}}\frac{\left(\varepsilon T^{{n}}\right)^{\frac{\sigma }{{m}+{n}}}}{\left(\varepsilon T^{{n}}\right)^{\frac{\sigma }{{m}+{n}}-1}}=\varepsilon T^{{n}}.\nonumber
\end{equation}

\subsection{Proof of Lemma \ref{lem:exponentswelldef}}
\label{sec:exponentwelldef}

Throughout this section, we denote (once again) by $1\leq h_{s}\leq {n}$ the largest non-zero index $j$ such that $q^{s}_{j}\neq 0$, for $s=1,\dotsc,{m}+{n}-1$.

Fix an index $1\leq\sigma\leq m+n$. If $\sigma <{m}$, by definition, we have that 
$$k_{\sigma }:={m}+h_{\sigma },$$
and part $i)$ holds true. If $\sigma \geq {m}$, we simultaneously prove parts $i)$ and $ii)$ by recursion on $\sigma $. First, we let $\sigma ={m}$. Since
$$k_{{m}}:={m}+h_{{m}},$$
part $i)$ holds true. Further, observe that for $\sigma=m$
\begin{multline}
\label{eq:lem3.71}
\alpha_{\sigma}(\sigma+1)+\sum_{j>\sigma+1-m}\alpha_{\sigma_j} \\
=({m}+1)\sum_{s=1}^{{m}}\frac{1}{{m}+h_{s}}+\sum_{j=2}^{{n}}\sum_{s=1}^{{m}}\frac{t_{sj}}{{m}+h_{s}}=\sum_{s=1}^{{m}}\frac{{m}+1+\sum_{j=2}^{{n}}t_{sj}}{{m}+h_{s}}.
\end{multline}
Since $q_{1}^{s}\neq 0$ for $s=1,\dotsc,{m}$, we also have that
\begin{equation}
\label{eq:lem3.72}
1+\sum_{j=2}^{{n}}t_{sj}=h_{s}.
\end{equation}
Hence, Equations (\ref{eq:lem3.71}) and (\ref{eq:lem3.72}) imply part $ii)$. Now, let us take ${m}<\sigma \leq {m}+{n}-1$, and let us suppose that both parts $i)$ and $ii)$ hold for all the indices $s$ such that ${m}\leq s<\sigma $. By (\ref{eq:from40}), we have that
\begin{equation}
k_{\sigma }:=(h_{\sigma }+{m})\left(1-\left(\frac{1-t_{1(\sigma +1-{m})}}{k_{1}}+\dotsb+\frac{1-t_{(\sigma -1)(\sigma +1-{m})}}{k_{\sigma -1}}\right)\right)^{-1}.\nonumber
\end{equation}
Hence, to prove part $i)$ for $\sigma $, it suffices to show that
$$0\leq\frac{1-t_{1(\sigma +1-{m})}}{k_{1}}+\dotsb+\frac{1-t_{(\sigma -1)(\sigma +1-{m})}}{k_{\sigma -1}}<1.$$
Since either $t_{sj}=0$ or $t_{sj}=1$ for all $s$ and $j$, the recursive hypothesis for part $i)$ ($s<\sigma $) implies that
\begin{equation}
\label{eq:0}
\frac{1-t_{(\sigma +1-{m})1}}{k_{1}}+\dotsb+\frac{1-t_{(\sigma -1)(\sigma +1-{m})}}{k_{\sigma -1}}\geq0.
\end{equation}
To prove the other inequality, we observe that
\begin{equation}
\label{eq:0.5}
\frac{1-t_{(\sigma +1-{m})1}}{k_{1}}+\dotsb+\frac{1-t_{(\sigma -1)(\sigma +1-{m})}}{k_{\sigma -1}}\leq\frac{1}{k_{1}}+\dotsb+\frac{1}{k_{\sigma -1}}=\alpha_{\sigma -1}.
\end{equation}
Since the recursive hypothesis for part $ii)$ ($s=\sigma -1$) implies that
$$\alpha_{\sigma -1}\sigma +\sum_{j>\sigma -{m}}\alpha_{(\sigma -1)j}=\sigma -1,$$
and, since $\alpha_{(\sigma -1)j}\geq 0$ for all $j$, we deduce that
\begin{equation}
\label{eq:1}
\alpha_{\sigma -1}\leq\frac{\sigma -1}{\sigma }<1.
\end{equation}
Hence, combining (\ref{eq:0}), (\ref{eq:0.5}), and (\ref{eq:1}), we obtain that
$$0\leq\frac{1-t_{(\sigma +1-{m})1}}{k_{1}}+\dotsb+\frac{1-t_{(\sigma -1)(\sigma +1-{m})}}{k_{\sigma -1}}<1.$$
We are now left to prove part $ii)$ for $s=\sigma $. We start by observing that
\begin{equation}
\label{eq:finallemma1}
\alpha_{\sigma }(\sigma +1)+\sum_{j>\sigma +1-{m}}\alpha_{\sigma j}=\alpha_{\sigma -1}\sigma +\alpha_{\sigma -1}+\frac{1}{k_{\sigma }}(\sigma +1)+\sum_{j>\sigma +1-{m}}\alpha_{(\sigma -1)j}+\sum_{j>\sigma +1-{m}}\frac{t_{\sigma j}}{k_{\sigma }}.
\end{equation}
We claim that
\begin{equation}
\label{eq:finallemma2}
\alpha_{\sigma -1}+\frac{1}{k_{\sigma }}(\sigma +1)+\sum_{j>\sigma +1-{m}}\frac{t_{\sigma j}}{k_{\sigma }}=1+\alpha_{(\sigma -1)(\sigma +1-{m})}.
\end{equation}
This concludes the proof, since (\ref{eq:finallemma1}), (\ref{eq:finallemma2}), and the recursive hypothesis for part $ii)$ ($s=\sigma -1$) imply that
$$\alpha_{\sigma }(\sigma +1)+\sum_{j>\sigma +1-{m}}\alpha_{\sigma j}=\alpha_{\sigma -1}\sigma +\alpha_{(\sigma -1)(\sigma +1-{m})}+\sum_{j>\sigma +1-{m}}\alpha_{(\sigma -1)j}+1=\sigma -1+1.$$

Now, we prove (\ref{eq:finallemma2}). By assumption, $t_{\sigma j}=1$ for $1\leq j\leq\sigma +1-{m}$, hence, we have that
$$\frac{1}{k_{\sigma }}(\sigma +1)=\frac{1}{k_{\sigma }}\left({m}+\sum_{j=1}^{\sigma +1-{m}}t_{\sigma j}\right),$$
whence
\begin{equation}
\label{eq:explanation0}
\alpha_{\sigma -1}+\frac{1}{k_{\sigma }}(\sigma +1)+\sum_{j>\sigma +1-{m}}\frac{t_{\sigma j}}{k_{\sigma }}=\alpha_{\sigma -1}+\frac{1}{k_{\sigma }}\left({m}+\sum_{j=1}^{{n}}t_{\sigma j}\right)=\alpha_{\sigma -1}+\frac{{m}+h_{\sigma }}{k_{\sigma }}.
\end{equation}
Finally, we observe that (\ref{eq:fillingup}) for $s=\sigma $ can be rewritten as
\begin{equation}
\label{eq:explanation}
\alpha_{\sigma -1}-\alpha_{(\sigma -1)(\sigma +1-{m})}=1-\frac{{m}+h_{\sigma }}{k_{\sigma }}.
\end{equation}
Thus, combining (\ref{eq:explanation0}) and (\ref{eq:explanation}), we obtain that
$$\alpha_{\sigma -1}+\frac{1}{k_{\sigma }}(\sigma +1)+\sum_{j>\sigma +1-{m}}\frac{t_{\sigma j}}{k_{\sigma }}=1+\alpha_{(\sigma -1)(\sigma +1-{m})},$$
which proves (\ref{eq:finallemma2}).

\section{Proof of Theorem \ref{thm:BHV}}
\label{sec:proofbhv}

We work by induction on $n$. The case $n=1$ is \eqref{eq:kru}. For $n>1$ we observe that in the expression of the sum $S(\bs\alpha, \bs T)$, we may always assume that $q_1,\dotsc,q_n\neq 0$, since for the vectors $\bs q$ such that at least one $q_i=0$ the required estimate follows from the inductive hypothesis. Throughout the rest of this section we will assume, without loss of generality, that $\bs\alpha\in[0,1)^n$. Let us fix constants $C\geq 1$ and $\varepsilon_0>0$ and let us consider the following ranges for the sum $S(\bs \alpha,\bs T)$:\vspace{2mm}
\begin{align}
    \label{eq:range1}
    & q_{1}\dotsm q_{n}\|\bs q\cdot\bs\alpha\|\leq\log(T_1\dotsm T_n)^{-n-\varepsilon_0}, \\[2mm]
    & 
\label{eq:range2}
    \log(T_1\dotsm T_n)^{-n-\varepsilon_0}<q_{1}\dotsm q_{n}\|\bs q\cdot\bs\alpha\|\leq\log(T_1\dotsm T_n)^{C}, \\[2mm]
    & \label{eq:range3}
    \log(T_1\dotsm T_n)^{C}<q_{1}\dotsm q_{n}\|\bs q\cdot\bs\alpha\|\leq T_1\dotsm T_n.
\end{align}
\vspace{2mm}Then
$$S(\bs\alpha,\bs T)=S_{1}(\bs\alpha,\bs T)+S_{2}(\bs\alpha,\bs T)+S_{3}(\bs\alpha,\bs T)$$
with
\begin{align*}
    & S_{1}(\bs\alpha,\bs T):=\sum_{\substack{0< q_{i}\leq T_{i}, \\ (\ref{eq:range1})\mbox{ holds}}}\frac{1}{q_{1}\dotsm q_{n}\|\bs q\cdot\bs\alpha\|} \\
    & S_{2}(\bs\alpha,\bs T):=\sum_{\substack{0< q_{i}\leq T_{i}, \\ (\ref{eq:range2})\mbox{ holds}}}\frac{1}{q_{1}\dotsm q_{n}\|\bs q\cdot\bs\alpha\|} \\
    & S_{3}(\bs\alpha,\bs T):=\sum_{\substack{0< q_{i}\leq T_{i}, \\ (\ref{eq:range3})\mbox{ holds}}}\frac{1}{q_{1}\dotsm q_{n}\|\bs q\cdot\bs\alpha\|}.
\end{align*}
Analogously, for the sum $S^{*}(\bs\alpha,T)$, we consider the ranges
\begin{align}
    \label{eq:range1*}
    & q\|q\alpha_1\|\dotsm\|q\alpha_n\|\leq(\log T)^{-n-\varepsilon_0}, \\[2mm]
    & 
\label{eq:range2*}
    (\log T)^{-n-\varepsilon_0}<q\|q\alpha_1\|\dotsm\|q\alpha_n\|\leq(\log T)^{C}, \\[2mm]
    & \label{eq:range3*}
    (\log T)^{C}<q\|q\alpha_1\|\dotsm\|q\alpha_n\|\leq T,
\end{align}
\vspace{2mm}and we write
$$S^{*}(\bs\alpha,T)=S_{1}^{*}(\bs\alpha,T)+S_{2}^{*}(\bs\alpha,T)+S_{3}^{*}(\bs\alpha,T),$$
with
\begin{align*}
    & S_{1}^{*}(\bs\alpha,T):=\sum_{\substack{0< q\leq T, \\ (\ref{eq:range1*})\mbox{ holds}}}\frac{1}{q\|q\alpha_1\|\dotsm\|q\alpha_n\|} \\
    & S_{2}^{*}(\bs\alpha,T):=\sum_{\substack{0< q\leq T, \\ (\ref{eq:range2*})\mbox{ holds}}}\frac{1}{q\|q\alpha_1\|\dotsm\|q\alpha_n\|} \\
    & S_{3}^{*}(\bs\alpha,T):=\sum_{\substack{0<q<T, \\ (\ref{eq:range3*})\mbox{ holds}}}\frac{1}{q\|q\alpha_1\|\dotsm\|q\alpha_n\|}.
\end{align*}
In order to prove Theorem \ref{thm:BHV}, we will estimate each of these sums separately.

\subsection{Estimating the sums $S_{1}$ and $S_{1}^{*}$}

Let us consider the inequalities
\begin{equation}
\label{eq:BC}
q_{1}^{+}\dotsm q_{n}^{+}\|\bs q\cdot\bs\alpha\|\geq \left(\frac{1}{\log (q_{1}^{+}\dotsm q_{n}^{+})}\right)^{n+\varepsilon_0}
\end{equation}
and
\begin{equation}
\label{eq:BCdual}
q\|q\alpha_1\|\dotsm\|q\alpha_n\|\geq \left(\frac{1}{\log q}\right)^{n+\varepsilon_0},
\end{equation}
where $\cdot$ stands for the standard dot-product in $\mb R^n$.

The following result is an easy consequence of the Borel-Cantelli Lemma.

\begin{lem}\label{lem:BC}
 For almost every $\bs\alpha\in\mb [0,1)^{n}$ there are only finitely many integers $q_1,\dotsc,q_n$ such that the converse of (\ref{eq:BC}) holds and finitely many integers $q$ such that the converse of (\ref{eq:BCdual}) holds.
\end{lem}

\begin{proof}
Fix $q_1,\dotsc,q_n\in\mb Z^{n}\setminus\{\bs 0\}$. The set
$$\bigcup_{p\in\mb Z}\left\{\bs\alpha\in[0,1)^n:|q_1\alpha_1+\dotsb +q_n\alpha_n +p|<\frac{1}{q_1^+\dotsm q_n^+\log\left(q_1^+\dotsm q_n^+\right)^{n+\varepsilon_0}}\right\}$$
has volume bounded by $\left(q_1^+\dotsm q_n^+\right)^{-1}\left(\log q_1^+\dotsm q_n^+\right)^{-n-\varepsilon_0}$, which is a summable function of $q_1,\dotsc,q_n$. Hence, the result for (\ref{eq:BC}) follows from the Borel-Cantelli Lemma. For (\ref{eq:BCdual}), the computation is analogous. 
\end{proof}

From Lemma \ref{lem:BC}, we deduce that for almost every $\bs\alpha\in\mb [0,1)^{n}$
\begin{equation}
\label{eq:S1}
S_{1}(\bs{\alpha},\bs T)=O_{\bs{\alpha}}(1)\quad\mbox{and}\quad S_{1}^{*}(\bs\alpha,T)=O_{\bs\alpha}(1).
\end{equation}
In view of this, we only need to estimate the sums $S_2,S_2^{*},S_3,$ and $S_3^{*}$. We first proceed to estimate the sums $S_3$ and $S_3^{*}$, which are more sensitive to a geometric approach.

\subsection{Estimating the sums $S_3$ and $S_3^{*}$}
For $\bs{\alpha}\in[0,1)^{n}$, $\bs T\in[1,+\infty)^{n}$, $T\geq 1$, and $0<a<b$, we introduce the following counting functions:
$$N(\bs{\alpha},\bs T,a,b):=\#\left\{\bs q\in[0,T_{1}]\times\dotsm\times[0,T_{n}]\cap\mb{Z}^{n}:a<q_{1}\dotsm q_{n}\|\bs q\cdot\bs\alpha\|\leq b\right\}$$
and
$$N^{*}(\bs{\alpha},T,a,b):=\#\left\{q\in[0,T]\cap\mb{Z}:a<q\|q\alpha_1\|\dotsm\|q\alpha_n\|\leq b\right\}.$$
Then, we have
\begin{equation}
\label{eq:finalform}
S_{3}(\bs{\alpha},\bs T)\ll_{\bs{\alpha}} \sum_{k=-\log\left(T_1\dotsm T_n\right)}^{-C\log\log\left (T_1\dotsm T_n\right)}e^{k+1}N\left(\bs{\alpha},\bs T,e^{-k-1},e^{-k}\right),
\end{equation}
and
\begin{equation}
\label{eq:finalformdual}
S_{3}^{*}(\bs{\alpha}, T)\ll_{\bs{\alpha}} \sum_{k=-\log T}^{-C\log\log T}e^{k+1}N^{*}\left(\bs{\alpha}, T,e^{-k-1},e^{-k}\right).
\end{equation}

We will show that the following result holds.
\begin{lem}
\label{lem:lp}
For almost every $\bs\alpha\in\mb [0,1)^n$, all $\bs T\in[1,+\infty)^n$, and all $C\geq 1$ we have that
$$S_{3}(\bs\alpha,\bs T)\ll_{n,\bs\alpha,C} \log \bar T\log T_1\dotsm\log T_{n}+(\log \bar T)^{n(n+\varepsilon_0)/(n+1)+1-C/(n+1)},$$
where $\bar T:=(T_1\dotsm T_n)^{1/n}$.
\end{lem}

\begin{proof}
To prove Lemma \ref{lem:lp}, we write
\begin{align}
\label{eq:Martineq}
& N\left(\bs{\alpha},\bs T,e^{-k-1},e^{-k}\right)\leq \nonumber  \\
& \sum_{\substack{h_i\leq \log T_i\ i=1,\dotsc,n \\ h_1+\dotsb+h_{n}\geq -k}}\#\left\{q_1\leq e^{h_1},\dotsc,q_{n}\leq e^{h_{n}}: e^{-k-h_1-\dotsb -h_{n}-1}\leq\|\bs q\cdot\bs\alpha\|<e^{-k-h_1-\dotsb -h_{n}+n}\right\}\nonumber \\
& \leq \sum_{\substack{h_i\leq \log T_i\ i=1,\dotsc,n \\ h_1+\dotsb+h_{n}\geq -k}}\# M\left(\bs{\alpha}^t,e^{-k-h_1-\dotsb -h_{n}+n},1/2,e^{h_1},\dotsc,e^{h_{n}}\right),
\end{align}
where the condition $h_1+\dotsb+h_{n}\geq -k$ follows from the fact that
$$e^{-h_1-\dotsb -h_{n}-k-1}\leq \|\bs q\cdot\bs\alpha\|<1.$$

By Lemma \ref{lem:BC}, for almost every $\bs{\alpha}\in\mb [0,1)^{n}$ there exists a constant $c_{\bs\alpha}>0$ such that
$$q_1^{+}\dotsm q_{n}^{+}\|\bs q\cdot\bs\alpha\|\geq c_{\bs\alpha}\log\left(q_1^{+}\dotsm q_{n}^{+}\right)^{-n-\varepsilon_0}.$$
Then almost every $\bs\alpha^t\in[0,1)^{n}$ is $d_{\bs{\alpha}}(\log x)^{-n-\varepsilon_0}$-multiplicatively badly approximable, where $d_{\bs\alpha}$ is a constant only depending on $c_{\bs\alpha}$ and $n$. By Theorem \ref{prop:cor1}, with $R=1/2$, we deduce that
\begin{multline}\# M\left(\bs{\alpha}^t,e^{-k-h_1-\dotsb -h_{n}+n},1/2,e^{h_1},\dotsc,e^{h_{n}}\right)\\
\ll e^{-k}+\left(d_{\bs\alpha}^{-1}e^{-k}(h_1+\dotsb+ h_{n})^{n+\varepsilon_0}\right)^{n/(n+1)}.
\end{multline}
Substituting into (\ref{eq:Martineq}) (note that $k<0$ by (\ref{eq:range3})), one obtains
\begin{multline*}
N\left(\bs{\alpha},\bs T,e^{-k-1},e^{-k}\right) \\
\ll_{n} e^{-k}\log T_1\dotsm\log T_{n}+d_{\bs\alpha}^{-n/(n+1)}e^{-nk/(n+1)}(\log \bar T)^{n(n+\varepsilon_0)/(n+1)+1},    
\end{multline*}
whence, by (\ref{eq:finalform}),
$$S_{3}(\bs\alpha,\bs T)\ll_{n,\bs\alpha,C} \log \bar T\log T_1\dotsm\log T_{n}+(\log \bar T)^{n(n+\varepsilon_0)/(n+1)+1-C/(n+1)},$$
proving the desired estimate.
\end{proof}

We are now left to estimate the sum $S_3^{*}$.

\begin{lem}
\label{lem:lpdual}
For almost every $\bs\alpha\in\mb [0,1)^n$, all $T\geq 1$, and all $C\geq 1$ we have that
$$S_{3}^{*}(\bs\alpha, T)\ll_{n,\bs\alpha,C} (\log T)^{n+1}+(\log T)^{n(n+\varepsilon_0)/(n+1)+n-C/(n+1)}.$$
\end{lem}

\begin{proof}
We write
\begin{align}
\label{eq:Martineqdual}
& N^{*}\left(\bs{\alpha}, T,e^{-k-1},e^{-k}\right)\leq \nonumber  \\
& \sum_{-k\leq h\leq \log T}\#\left\{q\leq e^{h}: e^{-k-h-1}\leq\|q\alpha_1\|\dotsm\|q\alpha_n\|<e^{-k-h+1}\right\}\nonumber \\
& \leq \sum_{-k\leq h\leq \log T}\# M\left(\bs{\alpha},e^{-k-h+1},1/2,e^h\right),
\end{align}
where the condition $h\geq -k$ derives from the fact that $\|q\alpha_1\|\dotsm\|q\alpha_n\|<1$.

By Lemma \ref{lem:BC}, for almost every $\bs{\alpha}\in\mb [0,1)^{n}$ there exists a constant $c_{\bs\alpha}>0$ such that
$$q\|q\alpha_1\|\dotsm\|q\alpha_n\|\geq c_{\bs\alpha}\left(\log q\right)^{-n-\varepsilon_0}.$$
Then, almost every $\bs\alpha\in[0,1)^{n}$ is $c_{\bs{\alpha}}(\log x)^{-n-\varepsilon_0}$-multiplicatively badly approximable. By Theorem \ref{prop:cor1}, with $R=1/2$, we deduce that
\begin{equation}\# M\left(\bs{\alpha},e^{-k-h+1},1/2,e^{h}\right)
\ll |k+h|^{n-1} e^{-k}+|k+h|^{n-1}\left(c_{\bs\alpha}e^{-k}h^{n+\varepsilon_0}\right)^{n/(n+1)}.\nonumber
\end{equation}
Substituting into (\ref{eq:Martineqdual}), we find
$$N^{*}\left(\bs{\alpha}, T,e^{-k-1},e^{-k}\right)\ll e^{-k}(\log T)^{n}+c_{\bs\alpha}^{n/(n+1)}e^{-nk/(n+1)}(\log T)^{n(n+\varepsilon_0)/(n+1)+n},$$
whence, by (\ref{eq:finalformdual}),
$$S_{3}^{*}(\bs\alpha, T)\ll_{n,\bs\alpha,C} (\log T)^{n+1}+(\log T)^{n(n+\varepsilon_0)/(n+1)+n-C/(n+1)}.$$
\end{proof}

\subsection{Completion of Proof}

Lemmas \ref{lem:lp} and \ref{lem:lpdual} show that, on choosing $C\gg_{n} 1$, the sums $S_3(\bs\alpha,\bs T)$ and $S_3^{*}(\bs\alpha,T)$ are bounded above by the functions $\log \bar T\log T_1\dotsm\log T_n$ and $(\log T)^{n+1}$ respectively.  We are then left to study the sums $S_2$ and $S_2^{*}$. For these sums we have the following estimates.

\begin{lem}
\label{lem:spmainbody}
For almost every $\bs\alpha\in\mb [0,1)^n$, all $\bs T\in[1,+\infty)^n$, $C\geq 1$, and $\eta>0$ it holds that
$$S_{2}(\bs\alpha,\bs T)\ll_{\bs\alpha,\eta,n,C} (\log\log \bar T)^{(n+1)(2+\eta)+1}\log T_1\dotsm\log T_{n},$$
where $\bar T:=(T_1\dotsm T_n)^{1/n}$.
\end{lem}

\begin{lem}
\label{lem:spdualmainbody}
For almost every $\bs\alpha\in\mb [0,1)^n$, all $T\geq 1$, $C\geq 1$, and $\eta>0$ it holds that
$$S_{2}^{*}(\bs\alpha, T)\ll_{n,\bs\alpha,C,\eta} (\log\log T)^{n+2+\eta}(\log T)^n.$$
\end{lem}

The exponents $(n+1)(2+\eta)+1$ and $n+2+\eta$ are likely not optimal in this case. The correct factors here should instead be $\varphi(\log\log(T_1\dotsm T_n))$ and $\varphi(\log\log T)$ for any function $\varphi:[1,+\infty)\to\mb (0,+\infty)$ such that $\sum_{n}\varphi(n)^{-1}<+\infty$ (compare with \cite[Lemma 4.1]{Bec94}).

Combining (\ref{eq:S1}) with Lemmas \ref{lem:lp}, \ref{lem:lpdual}, \ref{lem:spmainbody}, and \ref{lem:spdualmainbody} we deduce (\ref{eq:BHV}) and (\ref{eq:BHV*}). Lemmas \ref{lem:spmainbody} and \ref{lem:spdualmainbody} will be proved in Section \ref{sec:fmm}.

\appendix

\section{Estimating the Sums $S_2$ and $S_2^{*}$}
\label{sec:fmm}
\smallskip
\begin{center}by \textsc{Michael Bj\"orklund, Reynold Fregoli,} and \textsc{Alexander Gorodnik}\end{center}
\medskip

In this appendix we will be proving Lemmas \ref{lem:spmainbody} and \ref{lem:spdualmainbody}. For the reader's convenience, we recall the notation and the statement of the lemmas below. Here and throughout $\|x\|$ stands for distance form $x\in\mb R$ to the nearest integer, while $\cdot$ stands for the usual dot product in $\mb R^n$.

For $\bs\alpha\in[0,1)^n$ and $\bs T\in[1,+\infty)$ let
$$S(\bs\alpha,\bs T):=\sum_{0< q_{i}\leq T_{i}}\frac{1}{q_{1}\dotsm q_{n}\|\bs q\cdot\bs\alpha\|}$$
and for $T\geq 1$ let
$$S^{*}(\bs\alpha,T):=\sum_{0< q\leq T}\frac{1}{q\|q\alpha_1\|\dotsm\|q\alpha_n\|}.$$
Fix $\varepsilon_0>0$ and $C\geq 1$ and consider the following inequalities
\begin{equation}
\label{eq:range2app}
\log(T_1\dotsm T_n)^{-n-\varepsilon_0}<q_{1}\dotsm q_{n}\|\bs q\cdot\bs\alpha\|\leq\log(T_1\dotsm T_n)^{C}
\end{equation}
and
\begin{equation}
\label{eq:range2starapp}
(\log T)^{-n-\varepsilon_0}<q\|q\alpha_1\|\dotsm\|q\alpha_n\|,\leq(\log T)^{C}.
\end{equation}
Define
$$S_{2}(\bs\alpha,\bs T):=\sum_{\substack{0< q_{i}\leq T_{i}, \\ (\ref{eq:range2app})\mbox{ holds}}}\frac{1}{q_{1}\dotsm q_{n}\|\bs q\cdot\bs\alpha\|}$$
and
$$S_{2}^{*}(\bs\alpha,T):=\sum_{\substack{0< q\leq T, \\ (\ref{eq:range2starapp})\mbox{ holds}}}\frac{1}{q\|q\alpha_1\|\dotsm\|q\alpha_n\|}.$$
We aim to prove the following two lemmas.
\begin{lem}
\label{lem:range2app}
For almost every $\bs\alpha\in\mb [0,1)^n$, all $\bs T\in[1,+\infty)^n$, $C\geq 1$, and $\eta>0$ it holds that
$$S_{2}(\bs\alpha,\bs T)\ll_{\bs\alpha,\eta,n,C} (\log\log \bar T)^{(n+1)(2+\eta)+1}\log T_1\dotsm\log T_{n},$$
where $\bar T:=(T_1\dotsm T_n)^{1/n}$.
\end{lem}

\begin{lem}
\label{lem:range2starapp}
For almost every $\bs\alpha\in\mb [0,1)^n$, all $T\geq 1$, $C\geq 1$, and $\eta>0$ it holds that
$$S_{2}^{*}(\bs\alpha, T)\ll_{n,\bs\alpha,C,\eta} (\log\log T)^{n+2+\eta}(\log T)^n.$$
\end{lem}

The proof rests on two key Propositions, i.e., Propositions \ref{prop:Schmidt} and \ref{prop:Schmidtdual}, which are an adaptation of a well-known argument of Schmidt \cite{Sch60}. The interested reader is directed to \cite{KSW17} for a detailed exposition of a similar approach in a different setting.

\subsection{Proof of Lemma \ref{lem:range2app}}

We start by observing that
$$S_{2}(\bs\alpha,\bs T)\ll\sum_{k=-C\log\log(T_1\dotsm T_n)}^{(n+\varepsilon_0)\log\log (T_1\dotsm T_n)}e^{k+1}N\left(\bs{\alpha},\bs T,e^{-k-1},e^{-k}\right),$$
where for $\bs{\alpha}\in[0,1)^{n}$, $\bs T\in[1,+\infty)^{n}$, and $0<a<b$ we put
$$N(\bs{\alpha},\bs T,a,b):=\#\left\{\bs q\in[0,T_{1}]\times\dotsm\times[0,T_{n}]\cap\mb{Z}^{n}:a<q_{1}\dotsm q_{n}\|\bs q\cdot\bs\alpha\|\leq b\right\}.$$

We now require a result from Section \ref{sec:tess} of the main body of this paper, which we recall below. For $\varepsilon>0$, $\bs T\in[1,+\infty)^{n}$, and $0<R\leq 1$ consider the set
$$H_2:=\left\{\bs{x}\in\mb{R}^{n+1}: \prod_{i=0}^{n}|x_i|\leq \varepsilon,\ |x_0|\leq R,\ 1\leq |x_{i}|\leq T_i,\ i=1,\dotsc,n\right\}$$
and define
$$H_{2+}:=H_2\cap\left\{\bs{x}\in\mb{R}^{n+1}:x_{i}\neq 0,\ i=0,\dotsc,{n}\right\}.$$
In Section \ref{sec:tess} we proved the following statement.
\begin{prop}
\label{prop:tessapp}
Let $\varepsilon,\bs T$, and $R$ as above, and assume that $\varepsilon<RT_1\dotsm T_n$. Then there exist a set of indices $J$, a covering $H_{2+}\subset\bigcup_{\beta\in{J}}\! Y_{\beta}$ of the set $H_{2+}$, and a collection of linear maps $\left\{\psi_{\beta}\right\}_{\beta\in{{J}}}$ from $\mb{R}^{{n}+1}$ to itself, such that\vspace{2mm}
\begin{itemize}
    \item[$i)$] $ J=\left([0,\log T_1]\times\dotsb\times[0,\log T_n]\cap\mb Z^{n}\right)\times\{1,\dotsc,2^{n}\}$ (in particular, $J$ is independent of the choice of $\varepsilon$);\vspace{2mm}
    \item[$ii)$] the maps $\psi_{\beta}$ for $\beta\in{J}$ are determined by the expressions $\psi_{\beta}(\bs{x})_{i}:=\pm e^{b_{\beta,i}}\cdot x_{i}$ for $i=0,\dotsc,{n}$ and the coefficients $b_{\beta,i}$ satisfy\vspace{2mm}
    \begin{itemize}
        \item[$iia)$] $b_{\beta,0}\geq 0$ and $b_{\beta,1},\dots,b_{\beta,n}\leq 0$;\vspace{2mm}
        \item[$iib)$] $\sum_{i=0}^{n}b_{\beta,i}=0$;
    \end{itemize}
    \vspace{2mm}
    \item[$iii)$] the sets $Y_\beta$ are measurable and $\psi_\beta(Y_\beta)\subset (0,\varepsilon]\times[1,e]^{n}$ for all $\beta\in J$.
\end{itemize}
\end{prop}
Since the parameter $\varepsilon$ may change, let us write $H_2^{\varepsilon}$ and $H_{2+}^{\varepsilon}$ from now on, in place of $H_2$ (here and throughout $R=1/2$). Consider the lattice
\begin{equation*}
\Lambda_{\bs\alpha^t}:=
\begin{pmatrix}
1 & \bs\alpha^{t} \\
\bs{0} & I_{n}
\end{pmatrix}\mb{Z}^{n+1}
\end{equation*}
and note that for any $0<a\leq b$ it holds that
\begin{equation*}
N\left(\bs{\alpha},\bs T,a,b\right)\leq\#\left(\Lambda_{\bs\alpha^t}\cap H_2^b\right).
\end{equation*}
Let $H_{2+}^k:=H_{2+}^{e^{-k}}$ and denote by $Y_k$ and $\psi_{k}$ the covering and the maps deriving from Proposition \ref{prop:tessapp} for this set. Note that the set $J$ is independent of $k$. Then, by Proposition \ref{prop:tessapp}, it follows that
\begin{equation}
\label{eq:smallprod1}
S_{2}(\bs\alpha,\bs T)\leq \sum_{\beta\in J}\sum_{k=-C\log\log(T_1\dotsm T_n)}^{(n+\varepsilon_0)\log\log(T_1\dotsm T_n)}e^{k+1}\#\left(\psi_\beta^{k}\Lambda_{\bs\alpha^t}\cap \psi_\beta^{k}\left(Y_\beta^{k}\right)\right).
\end{equation}

We now make the following observation.
\begin{lem}
\label{lem:L1formula}
With the notation of Proposition \ref{prop:tessapp}, for all $\beta\in J$ we have that
$$\int_{[0,1)^{n}}\#\left(\psi_\beta\Lambda_{\bs\alpha^t}\cap \psi_\beta\left(Y_\beta\right)\right)d\bs\alpha\ll \varepsilon.$$
\end{lem}

\begin{proof}
Let $\chi$ denote the characteristic function of the set $(0,\varepsilon]\times[1,e]^{n}$. Then, by Proposition \ref{prop:tess}, we have that\footnote{We are considering only the cases where the diagonal map $\psi_{\beta}$ lies in the connected component of the identity in $\textup{SL}_{n+1}(\mb R)$, but the other cases are analogous.}
\begin{align}
& \int_{[0,1)^{n}}\#\left(\psi_\beta\Lambda_{\bs\alpha^t}\cap \psi_\beta\left(Y_\beta\right)\right)d\bs\alpha\nonumber \\
& \leq\sum_{p,q_1,\dotsc,q_{n}\in\mb Z}\int_{[0,1)^{n}}\chi\left(e^{b_{\beta,0}}(p+\alpha_1 q_1+\dotsb +\alpha_{n} q_{n}),e^{b_{\beta,1}}q_1,\dotsc,e^{b_{\beta,n}}q_{n}\right)d\bs\alpha\nonumber \\
& =\sum_{q_1,\dotsc,q_{n}\in\mb Z}\sum_{k\in\mb Z}\sum_{p_0=0}^{q_1-1}\int_{[0,1)^{n}}\chi\left(e^{b_{\beta,0}}((\alpha_1+k)q_1+p_0+\alpha_2 q_2\dotsb +\alpha_{n} q_{n}),e^{b_{\beta,1}}q_1,\dotsc,e^{b_{\beta,n}}q_{n}\right)d\bs\alpha\nonumber\\
& =\sum_{q_1,\dotsc,q_{n}\in\mb Z}\sum_{p_0=0}^{q_1-1}\int_{[0,1)^{n-1}}\int_{\mb R}\chi\left(e^{b_{\beta,0}}(xq_1+p_0+\dotsb +\alpha_{n} q_{n}),e^{b_{\beta,1}}q_1,\dotsc,e^{b_{\beta,n}}q_{n}\right)dx d\alpha_2\dotsm d\alpha_{n}\nonumber.
\end{align}
By the change of variables $y=e^{b_{\beta,0}}(xq_1+p_0+\dotsb +\alpha_{n} q_{n})$, we find
\begin{multline}
 \int_{[0,1)^{n}}\#\left(\psi_\beta\Lambda_{\bs\alpha^t}\cap \psi_\beta\left(Y_\beta\right)\right)d\bs\alpha \\
 \leq \sum_{q_1,\dotsc,q_{n}\in\mb Z}\sum_{p_0=0}^{q_1-1}\frac{1}{q_1e^{b_{\beta,0}}}\int_{[0,1)^{n-1}}\int_{\mb R}\chi\left(y,e^{b_{\beta,1}}q_2,\dotsc,e^{b_{\beta,n}}q_{n}\right)dy d\alpha_2\dotsm d\alpha_{n}.\nonumber
\end{multline}
Now, since $b_{\beta,i}< 0$ for $i=1,\dotsc,n$ (part $iia$ of Proposition \ref{prop:tessapp}), the term $\chi\left(y,e^{b_{\beta,1}}q_2,\dotsc,e^{b_{\beta,n}}q_{n}\right)$ is non-null only for $q_i\ll e^{-b_{\beta,i}}$. Moreover, by the definition of $\chi$, we have that
$$\int_{\mb R}\chi\left(y,e^{b_{\beta,1}}q_2,\dotsc,e^{b_{\beta,n}}q_{n}\right)dy\leq \varepsilon$$
independently of the value of $q_1,\dotsc,q_{n}$. 
By Part $iib$ of Proposition \ref{prop:tessapp}, it follows that
$$\int_{[0,1)^{n}}\#\left(\psi_\beta\Lambda_{\bs\alpha^t}\cap \psi_\beta\left(Y_\beta\right)\right)d\bs\alpha\ll \varepsilon e^{-\sum_{i=0}^{n}b_{\beta,i}}=\varepsilon,$$
concluding the proof.
\end{proof}

From Lemma \ref{lem:L1formula}, we deduce that for all subsets $\tilde J\subset J$ it holds that
\begin{equation}
\label{eq:Schmidt'shypothesis}
\int_{[0,1)^{n}}\left|\sum_{\beta\in \tilde J}\sum_{k=-C\log\log(T_1\dotsm T_n)}^{(n+\varepsilon_0)\log\log(T_1\dotsm T_n)}e^{k+1}\#\left(\psi_\beta^{k}\Lambda_{\bs\alpha^t}\cap \psi_\beta^{k}\left(Y_\beta^{k}\right)\right)\right|d\bs\alpha\ll_{n,C} \log\log \bar T\#\tilde J.
\end{equation}

To conclude, we rely on the following Proposition, which we prove in Subsection \ref{sec:Schmidtdual}.

\begin{prop}
\label{prop:Schmidt}
Let $d\geq 1$ and let $(Y,\nu)$ be a probability space. For $\bs n\in\mb Z^d$ ($n_i\geq 0$) let $f_{\bs n}:Y\to \mb R$ be a family of measurable functions. Assume that for any choice of $0\leq A_i<B_i\leq N_{i}$ for $i=1,\dotsc,d$, with $N_i\geq 1$, it holds that
\begin{equation}
\label{eq:Schmidt}
\int_Y\left|\sum_{A_i\leq n_i\leq B_i\ i=1,\dotsc, d}f_{\bs n}(y)\right|d\nu\leq g(N_1\dotsm N_d)(B_1-A_1)\dotsm(B_d-A_d),
\end{equation}
where $g:(1,+\infty]\to[1,+\infty]$ is increasing and such that $g(2x)\ll 2g(x)$ for all $x\geq 1$.
Then, for almost every $y\in Y$, every $\eta>0$, and every $N_1,\dotsc,N_d\geq 1$ it holds that
$$\left|\sum_{n_i\leq N_i\ i=1,\dotsc,d}f_{\bs n}(y)\right|\ll_{y,\eta,d}g(N_1\dotsm N_d)\log(N_1\dotsm  N_d)^{d(2+\eta)}N_1\dotsm N_d.$$
\end{prop}

In view of (\ref{eq:Schmidt'shypothesis}), Proposition \ref{prop:Schmidt}, applied with $N_i=\log T_i$, $\bs n=\beta$, $g(x)\asymp_{n,C}\log x$, and
$$f_{\beta}(\bs\alpha)=\sum_{k=-C\log\log(T_1\dotsm T_n)}^{(n+\varepsilon_0)\log\log(T_1\dotsm T_n)}e^{k+1}\#\left(\psi_\beta^{k}\Lambda_{\bs\alpha^t}\cap \psi_\beta^{k}\left(Y_\beta^{k}\right)\right),$$
shows that for any $\eta>0$ and almost every $\bs\alpha\in[0,1)^{n}$
$$S_{2}(\bs\alpha,\bs T)\ll_{\bs\alpha,\eta,n,C}(\log\log \bar T)^{(n+1)(2+\eta)+1}\log T_1\dotsm\log T_{n}.$$
This proves Lemma \ref{lem:range2app}.

\subsection{Proof of Lemma \ref{lem:range2starapp}}

In this subsection we will be estimating the sum $S_{2}^{*}(\bs\alpha, T)$. As in the previous case, we start by observing that
\begin{equation}
\label{eq:spdual}
S_{2}^{*}(\bs\alpha, T)\ll\sum_{k=-C\log\log T}^{(n+\varepsilon_0)\log\log T}e^{k+1}N^{*}\left(\bs{\alpha},T,e^{-k-1},e^{-k}\right),
\end{equation}
where for $\bs{\alpha}\in[0,1)^{n}$, $T\geq 1$, and $0<a<b$ we put
$$N^{*}(\bs{\alpha},T,a,b):=\#\left\{q\in[0,T]\cap\mb{Z}:a<q\|q\alpha_1\|\dotsm\|q\alpha_n\|\leq b\right\}.$$

Let
\begin{equation*}
\Lambda_{\bs\alpha}:=
\begin{pmatrix}
I_n & \bs\alpha \\
\bs{0}^t & 1
\end{pmatrix}\mb{Z}^{n+1}.
\end{equation*}
and for $\varepsilon>0$, $T>0$, and $0<R\leq 1$ let
\begin{equation*}
H_1:=\left\{\bs{x}\in\mb{R}^{n}:\prod_{i= 1}^{n}\left|x_{i}\right|<\varepsilon,\ |x_{i}|\leq R,\ i=1,\dotsc,n\right\}.
\end{equation*}
Let also
$$H_{1+}:=H_1\cap\{\bs x\in\mb R^{n}:x_i\neq 0\}.$$
In Section \ref{sec:tess}, we proved the following result.
\begin{prop}
\label{prop:partitionapp}
Suppose that $R^{{m}}/\varepsilon>e^{{m}}$, where $e=2.71828\dots$ is the base of the natural logarithm. Then there exist a set of indices $I$, a partition $H_{1+}=\bigcup_{\beta\in{I}}\! X_{\beta}$ of the set $H_{1+}$, and a collection of linear maps $\left\{\varphi_{\beta}\right\}_{\beta\in{{I}}}$ from $\mb{R}^{{m}}$ to itself, such that\vspace{2mm}
\begin{itemize}
\item[$i)$] $\#{I}\ll_{{m}}\log\left(R/\varepsilon^{1/{m}}\right)^{{m}-1}$;\vspace{2mm}
\item[$ii)$] the maps $\varphi_{\beta}$ for $\beta\in{I}$ are determined by the expression $\varphi_{\beta}(\bs{x})_{i}:=e^{a_{\beta,i}}\cdot x_{i}$ for
$i=1,\dotsc,{m}$, where the coefficients $a_{\beta,i}\in\mb{R}$ satisfy\vspace{2mm}
\begin{itemize}
\item[$iia)$] $e^{a_{\beta,i}}\gg_{{m}}\varepsilon^{1/{m}}/R$ for $i=1,\dotsc,{m}$;\vspace{2mm}
\item[$iib)$] $\sum_{i=1}^{{m}}a_{\beta,i}=0$;
\end{itemize}
\vspace{2mm}
\item[$iii)$] the sets $X_{\beta}$ are measurable and $\varphi_{\beta}\left(X_{\beta}\right)\subset\left[-c\varepsilon^{1/{m}},c\varepsilon^{1/{m}}\right]^{{m}}$ for all $\beta\in{I}$, where $c$ is a constant only depending on $m$.\vspace{2mm}
\end{itemize}
\end{prop}
Since the parameter $\varepsilon$ may change, let us write $H_1^{\varepsilon}$ and $H_{1+}^{\varepsilon}$ from now on, in place of $H_1$ (here and throughout $R=1/2$).

We notice that
\begin{equation*}
N^{*}\left(\bs{\alpha},T,a,b\right)\leq\#\left(\Lambda_{\bs\alpha}\cap H_{1+}^b\times[1,T]\right),
\end{equation*}
and it easily follows that
\begin{equation}
\label{eq:eq}
S_{2}^{*}(\bs\alpha,T)\leq \sum_{0\leq h\leq \log T}\sum_{k=-C\log\log T}^{(n+\varepsilon_0)\log\log T}e^{k+1}\#\left(\Lambda_{\bs\alpha}\cap \left(H_{1+}^{h,k}\times\left[1,e^h\right]\right)\right),
\end{equation}
where $H_{1+}^{h,k}:=H_{1+}^{e^{-k-h+1}}$.

We now require the following result.
\begin{lem}
\label{lem:correctvolume}
Let $T'\geq 1$ and $\varepsilon>0$. Then
$$\int_{[0,1)^n}\#\left(\Lambda_{\bs\alpha}\cap \left(H_{1+}^{\varepsilon}\times\left[1,T'\right]\right)\right)d\bs\alpha\ll \log\left(\frac{R^n}{\varepsilon}\right)^{n-1}\varepsilon T'.$$
\end{lem}

\begin{proof}
For each map $\varphi_{\beta}$ in Proposition \ref{prop:partitionapp}, denote by $\tilde{\varphi}_{\beta}$ the map $\varphi_{\beta}\times \textup{id}$ from $\mb R^{n+1}$ to itself. Then, by Proposition \ref{prop:partitionapp}, we may write
\begin{align}
& \int_{[0,1)^n}\#\left(\Lambda_{\bs\alpha}\cap \left(H_{1+}^{\varepsilon}\times\left[1,T'\right]\right)\right)d\bs\alpha\nonumber \\
& =\sum_{\beta\in I}\int_{[0,1)^n}\#\left(\tilde{\varphi}_{\beta}\Lambda_{\bs\alpha}\cap \tilde{\varphi}_{\beta}\left(X_\beta\times\left[1,T'\right]\right)\right)d\bs\alpha\nonumber \\
& = \sum_{\beta\in I}\sum_{p_1,\dotsc,p_{n},q\in\mb Z}\int_{[0,1)^n}\chi(e^{a_{\beta,1}}(\alpha_1 q+p_1),\dotsc,e^{a_{\beta,n}}(\alpha_n q+p_n),q)d\bs\alpha,\label{eq:86}
\end{align}
where $\chi$ denotes the characteristic function of the set $\left[-c\varepsilon^{1/n},c\varepsilon^{1/n}\right]^n\times [1,T']$.
The expression in \eqref{eq:86} can be further expanded as
\begin{multline}\sum_{\beta\in I}\sum_{q\in\mb Z}\sum_{k_1,\dotsc,k_n\in\mb Z}\sum_{p_{1}^0,\dotsc,p_{n}^0=0}^{q-1}\int_{[0,1)^n}\chi(e^{a_{\beta,1}}((\alpha_1+k_1)q+p_{1}^0),\dotsc,e^{a_{\beta,n}}((\alpha_n+k_n) q+p_{n}^0),q)d\bs\alpha \\
=\sum_{\beta\in I}\sum_{q\in\mb Z}\sum_{p_{1}^0,\dotsc,p_{n}^0=0}^{q-1}\int_{\mb R^n}\chi(e^{a_{\beta,1}}(x_1 q+p_{1}^0),\dotsc,e^{a_{\beta,n}}(x_n q+p_{n}^0),q)d\bs x.
\end{multline}
By the change of variables $y_i:=e^{a_{\beta,i}}(x_i q+p_{i}^0)$ and part $iib$ of Proposition \ref{prop:partitionapp}, we deduce that
\begin{multline*}
\int_{[0,1)^n}\#\left(\Lambda_{\bs\alpha}\cap \left(H_{1}\times\left[1,T'\right]\right)\right)d\bs\alpha\leq \sum_{\beta\in I}\sum_{q\in\mb Z}\frac{q^{n}}{q^n e^{\sum_{i}a_{\beta,i}}}\int_{\mb R^n}\chi(y_1,\dotsc,y_n,q)d\bs y \\
\ll_n \# I\cdot\varepsilon T'\ll_{n} \log\left(\frac{R^n}{\varepsilon}\right)^{n-1}\varepsilon T',
\end{multline*}
concluding the proof.
\end{proof}

From Lemma \ref{lem:correctvolume} and (\ref{eq:eq}), we conclude that for each fixed $0\leq h\leq \log T$ it holds that
\begin{multline}
 \int_{[0,1)^n}\sum_{k=-C\log\log T}^{(n+\varepsilon_0)\log\log T}e^{k+1}\#\left(\Lambda_{\bs\alpha}\cap \left(H_{1}^{h,k}\times\left[1,e^h\right]\right)\right)d\bs\alpha \\
 \ll_n \sum_{k=-C\log\log T}^{(n+\varepsilon_0)\log\log T}e^{k+1}|k+h|^{n-1}e^{-k-h}\cdot e^{h} \ll_{n,C}(\log\log T)^n\cdot h^{n-1}.\nonumber  
\end{multline}
This, in turn, implies that for all $A\leq B-1<B\leq\log T$ we have that
\begin{multline}
 \int_{[0,1)^n}\left|\sum_{A< h\leq B}\sum_{k=-C\log\log T}^{(n+\varepsilon_0)\log\log T}e^{k+1}\#\left(\Lambda_{\bs\alpha}\cap \left(H_{1}^{h,k}\times\left[1,e^h\right]\right)\right)\right|d\bs\alpha \\
 \ll_{n,C}(\log\log T)^{n}(B^n-A^n),  
\end{multline}
where we used the fact that
$$\sum_{A\leq h\leq B}h^{n-1}\ll_n B^n-A^n$$
for all $A\leq B-1$.

The following Proposition is yet another variation on Schmidt's method, which will be proved in Subsection \ref{sec:Schmidtdual}.

\begin{prop}
\label{prop:Schmidtdual}
Let $(Y,\nu)$ be a probability space and let $r\geq 1$. For $n\in \mb Z$ ($n\geq 0$) let $f_{n}:Y\to \mb R$ be a family of measurable functions and assume that for any choice of $0\leq A<B\leq N$, with $N\geq 1$, it holds that
\begin{equation}
\label{eq:Schmidtdual}
\int_Y\left|\sum_{A\leq n\leq B}f_{\bs n}(y)\right|d\nu\leq g(N)(B^r-A^r),
\end{equation}
where $g:(1,+\infty]\to[1,+\infty]$ is increasing and such that $g(2x)\ll 2g(x)$ for all $x\geq 1$.
Then, for almost every $y\in Y$, for every $\eta>0$ and $N\geq 1$  it holds that
$$\left|\sum_{n\leq N}f_{ n}(y)\right|\ll_{y,\eta,d}g(N)(\log N)^{2+\eta}N^{r}.$$
\end{prop}

On applying Proposition \ref{prop:Schmidtdual} with $N=\log T$, $n=h$ (where $n$ is the index in Proposition \ref{prop:Schmidtdual}), $r=n$ (where $n$ is the dimension of the space $\mb R^n$ in Theorem \ref{thm:BHV}), $g(x)\asymp_{n,C}(\log x)^n$, and
$$f_h(\alpha)=\sum_{k=-C\log\log T}^{(n+\varepsilon_0)\log\log T}e^{k+1}\#\left(\Lambda_{\bs\alpha}\cap \left(H_{1}^{h,k}\times\left[1,e^h\right]\right)\right)$$
we obtain that that for any $\eta>0$ and almost every $\bs\alpha\in[0,1)^{n}$
$$S_{2}^{*}(\bs\alpha,T)\ll_{\bs\alpha,\eta,n,C}(\log\log T)^{n+2+\eta}(\log T)^{n}.$$
This proves Lemma \ref{lem:range2starapp}.

\subsection{Proof of Proposition \ref{prop:Schmidt}}
\label{sec:Schmidt}

For $s\in\mb N$ define
$$L_s:=\left\{\left(2^a b,2^a(b+1)\right]:a,b=0,1,2\dotsc,\mbox{ and } 2^a(b+1)<2^s\right\}.$$

\begin{lem}
\label{lem:A1}
For any $s_1,\dotsc, s_d\in\mb N$ we have that
$$\sum_{I_i\in L_{s_i}\ i=1,\dotsc,d}\int_{Y}\left|\sum_{\bs n\in I_1\times\dotsb\times I_d}f_{\bs n}(y)\right|\leq g(2^{s_1+\dotsb +s_d})s_1\dotsm s_d 2^{s_1+\dotsb +s_d}.$$
\end{lem}

\begin{proof}
From (\ref{eq:Schmidt}) it follows that
$$\int_{Y}\left|\sum_{\bs n\in I_1\times\dotsb\times I_d}f_{\bs n}(y)\right|\leq g(2^{s_1+\dotsb +s_d})\cdot |I_1|\dotsm|I_d|.$$
To conclude, it suffices to observe that
$$\sum_{a\leq s}\sum_{b\leq 2^{s-a}}\left(2^{a}(b+1)-2^{a}b\right)\ll s2^{s}.$$    
\end{proof}

Using base $2$ expansion, one can prove the following lemma (see also \cite[Lemma 1]{Sch60}).

\begin{lem}
\label{lem:A2}
Let $k,s\in\mb N$ with $k<2^s$. Then, the interval $[0,k]$ is covered by at most $s$ disjoint intervals in the family $L_s$.
\end{lem}

The combination of the previous two lemmas, allows us to prove the subsequent result.

\begin{lem}
\label{lem:A3}
For any $s_1,\dotsc,s_d\in\mb N$ and any $\eta>0$ there exists a subset $Y_{s_1,\dotsc,s_d,\eta}$ of $Y$ such that
\begin{itemize}
    \item[$i)$] $\nu(Y_{s_1,\dots,s_d,\eta})\leq (s_1\dotsm s_d)^{-1-\eta}$;\vspace{2mm}
    \item[$ii)$] for all $N_1,\dots,N_d\in\mb N$ with $N_i<2^{s_i}$ ($i=1,\dotsc,d$) and all $y\notin Y_{s_1,\dotsc,s_d,\eta}$ it holds that
    $$\left|\sum_{n_i\leq N_i\ i=1,\dotsc,d}f_{\bs n}(y)\right|\leq g(2^{s_1+\dotsb +s_d})(s_1\dotsm s_d)^{2+\eta}2^{s_1+\dotsb +s_d}.$$
\end{itemize}
\end{lem}

\begin{proof}
Let
\begin{multline}
Y_{s_1,\dotsc,s_d,\eta}:= \\
\left\{y\in Y: \sum_{I_i\in L_{s_i}\ i=1,\dotsc,d}\left|\sum_{\bs n\in I_1\times\dotsb\times I_d}f_{\bs n}(y)\right|\geq g(2^{s_1+\dotsb +s_d})(s_1\dotsm s_d)^{2+\eta}2^{s_1+\dotsb +s_d}\right\}.\nonumber
\end{multline}
By Lemma \ref{lem:A1} and Chebychev's inequality, we deduce that $\nu(Y_{s_1,\dotsc,s_d,\eta})\leq (s_1\dotsm s_d)^{1+\eta}$. Now, fix $N_i\leq 2^{s_i}$ for $i=1,\dotsc,d$. Then, by Lemma \ref{lem:A2}, for $i=1,\dotsc,d$ there exists a family of disjoint intervals $L(N_i)\subset L_{s_i}$ covering the interval $(0,N_i]$. It follows that for $y\notin Y_{s_1,\dotsc,s_d,\eta}$ one has that
\begin{multline}
\left|\sum_{n_i\leq N_i\ i=1,\dotsc,d}f_{\bs n}(y)\right|\leq \sum_{I_i\in L(N_i)\ i=1,\dotsc,d}\left|\sum_{\bs n \in I_1\times\dotsb\times I_d}f_{\bs n}(y)\right| \\
\leq \sum_{I_i\in L_{s_i}\ i=1,\dotsc,d}\left|\sum_{\bs n \in I_1\times\dotsb\times I_d}f_{\bs n}(y)\right|\leq g(2^{s_1+\dotsb +s_d})(s_1\dotsm s_d)^{2+\eta}2^{s_1+\dotsb +s_d}.\nonumber 
\end{multline}
\end{proof}

By part $(i)$ of Lemma \ref{lem:A3} and the Borel-Cantelli Lemma, for almost every $y\in Y$ there are only finitely many parameters $s_1,\dots,s_d$ for which $y\in Y_{s_1,\dotsc,s_d,\eta}$. Let $y\in Y$ with this property, and assume that $y\notin Y_{s_1,\dots,s_d,\eta}$ for all $s_1,\dots,s_d$ with $\max_i s_i\geq A(y)$. For fixed $N_1,\dotsc,N_d$ pick $s_i$ such that $2^{s_i-1}\leq N_i< 2^{s_i}$ for $i=1,\dotsc,d$. Then, by part $(ii)$ of Lemma \ref{lem:A3}, one has that
\begin{multline}
\left|\sum_{n_i\leq N_i\ i=1,\dotsc,d}f_{\bs n}(y)\right|\leq g(2^{s_1+\dotsb +s_d})(s_1\dotsm s_d)^{2+\eta}2^{s_1+\dotsb +s_d} \\
+\underbrace{\max_{N_i'\leq 2^{A(y)}}\left|\sum_{n_i\leq N_i'\ i=1,\dotsc,d}f_{\bs n}(y)\right|}_{c(y)}.\nonumber
\end{multline}
Since $g(2x)\ll 2g(x)$ for all $x\geq 1$, we find that for all $N_1,\dotsc,N_d\geq 1$ it holds
$$\left|\sum_{n_i\leq N_i\ i=1,\dotsc,d}f_{\bs n}(y)\right|\ll \max\{1,c(y)\}\cdot g(N_1\dotsm N_d)\log(N_1\dotsm N_d)^{d(2+\eta)}N_1\dotsm N_d.$$

\subsection{Proof of Proposition \ref{prop:Schmidtdual}}
\label{sec:Schmidtdual}

In this subsection, we will once again use the sets $L_s$ introduced in Subsection \ref{sec:Schmidt}.

\begin{lem}
\label{lem:A1dual}
Let $s\in\mb N$. Then we have that
$$\sum_{I\in L_{s}}\int_{Y}\left|\sum_{n\in I}f_{n}(y)\right|\leq g(2^{s})s2^{sr}.$$
\end{lem}

\begin{proof}
By (\ref{eq:Schmidtdual}), for any interval $I$ of the form $I=(A,B]$ we have that
$$\int_{Y}\left|\sum_{n\in I}f_{n}(y)\right|\leq g(2^s)(B^{r}-A^r).$$
To conclude, it suffices to observe that
$$\sum_{a\leq s}\sum_{b\leq 2^{s-a}}\left(2^{ar}(b+1)^r-2^{ar}b^r\right)\ll_r s2^{sr}.$$    
\end{proof}

\begin{lem}
\label{lem:A3dual}
For any $s\in\mb N$ and any $\eta>0$ there exists a subset $Y_{s,\eta}$ of $Y$ such that
\begin{itemize}
    \item[$i)$] $\nu(Y_{s,\eta})\leq s^{-1-\eta}$;\vspace{2mm}
    \item[$ii)$] for all $N\in\mb N$ with $N<2^{s}$ and all $y\notin Y_{s,\eta}$ it holds that
    $$\left|\sum_{n\leq N}f_{n}(y)\right|\leq g(2^{s})s^{2+\eta}2^{sr}.$$
\end{itemize}
\end{lem}

\begin{proof}
Let
$$Y_{s,\eta}:=\left\{y\in Y:\sum_{I\in L_s}\left|\sum_{n\in I}f_{\bs n}(y)\right|> g(2^{s})s^{2+\eta}2^{sr} \right\}.$$
By Lemma \ref{lem:A1dual} and Chebyshev's Inequality we have that $\nu(Y_{s,\eta})\leq s^{-1-\eta}$, as required. Now, fix $N< 2^s$. By Lemma \ref{lem:A2}, the interval $[0,N]$ may be covered by a sub-collection of intervals $L(N)\subset L_s$. It follows that for $y\notin Y_{s,\eta}$ we have
\begin{equation*}
\left|\sum_{n\leq N}f_{n}(y)\right|\leq \sum_{I\in L(N)}\left|\sum_{n\in I}f_{n}(y)\right| \leq \sum_{I\in L_s}\left|\sum_{n\in I}f_{n}(y)\right|\leq g(2^s)s^{2+\eta}2^{sr}. 
\end{equation*}
\end{proof}

To conclude the proof of Proposition \ref{prop:Schmidtdual}, we observe that by the Borel-Cantelli Lemma, for almost every $y\in Y$ and any fixed $\eta>0$ there are only finitely $s$ such that $y\notin Y_{s,\eta}$. Pick $y$ with this property and assume that $y\notin Y_{s,\eta}$ for all $s\geq A(y)$. Fix $N\geq 1$ and $s$ such that $2^{s-1}\leq N<2^{s}$. Then we have that
$$\left|\sum_{n\leq N}f_{n}(y)\right|\leq g(2^s)s^{2+\eta}2^{sr}+\underbrace{\max_{N'\leq 2^{A(y)}}\left|\sum_{n\leq N'}f_{n}(y)\right|}_{c(y)},$$
whence, by the properties of $g$, we deduce that
$$\left|\sum_{n\leq N}f_{\bs n}(y)\right|\leq \max\{1,c(y)\}\cdot g(N)(\log N)^{2+\eta}N^r.$$

\section{Proof of Theorem \ref{thm:LV}}
\label{sec:proofofthmLV}

In what follows, all the sums will be over $\bs q\neq \bs 0$. We start by noticing that
\begin{align}
 & \sum_{\substack{0\leq q_i\leq T_i \\ i=1,\dotsc,n}}\prod_{i=1}^{m}\|L_{i}\bs{q}\|^{-1}\nonumber \\
 & \leq\sum_{k=0}^{\infty}2^{k+1}\#\left\{\bs{q}\in\prod_{j=1}^{n}[-T_{j},T_{j}]\cap\mb{Z}^{n}\setminus\{\bs{0}\}:2^{-k-1}\leq\prod_{i=1}^{m}\|L_{i}\bs{q}\|<2^{-k}\right\} \nonumber \\
 & \leq\sum_{k=0}^{\infty}2^{k+1}\#\left\{\bs{q}\in\prod_{j=1}^{n}[-T_{j},T_{j}]\cap\mb{Z}^{n}\setminus\{\bs{0}\}:\prod_{i=1}^{m}\|L_{i}\bs{q}\|<2^{-k}\right\}.\nonumber
\end{align}

From this, (\ref{eq:intersection}), and Lemma \ref{lem:emptycase} with $\varepsilon=2^{-k}$, we deduce that

\begin{align}
 \label{eq:cor2eq1} 
 \sum_{\substack{0\leq q_i\leq T_i \\ i=1,\dotsc,n}}\prod_{i=1}^{m}\|L_{i}\bs{q}\|^{-1} & \leq\sum_{k=0}^{\infty}2^{k+1}\# M\left(\bs{L},2^{-k},\frac{1}{2},\bs{T}\right)\nonumber \\
 & =\sum_{k=0}^{\left\lfloor\log_{2}\left(\frac{\bar T^{n}}{\phi(\bar T)}\right)\right\rfloor}2^{k+1}\# M\left(\bs{L},2^{-k},\frac{1}{2},\bs{T}\right).
\end{align}

We use Theorem \ref{prop:cor1} to estimate the right-hand side of (\ref{eq:cor2eq1}). We need $T^{m}/\varepsilon\geq e^{m}$, i.e., $2^{k-m}\geq e^{m}$. To ensure this condition, we split the sum in (\ref{eq:cor2eq1}) into two parts, one for $2^{k-m}<e^{m}$ and one for $2^{k-m}\geq e^{m}$. We find that
\begin{align}
 & \sum_{\substack{0\leq q_i\leq T_i \\ i=1,\dotsc,n}}\prod_{i=1}^{m}\|L_{i}\bs{q}\|^{-1}\leq\sum_{k=0}^{\left\lfloor m\left(1+1/\log 2\right)\right\rfloor}2^{k+1}\# M\left(\bs{L},2^{-k},\frac{1}{2},\bs{T}\right)\nonumber \\
 & +\sum_{k=\left\lceil m\left(1+1/\log 2\right)\right\rceil}^{\left\lfloor\log_{2}\left(\frac{\bar T^{n}}{\phi(\bar T)}\right)\right\rfloor}2^{k+1}\# M\left(\bs{L},2^{-k},\frac{1}{2},\bs{T}\right)\nonumber \\
 & \ll_{m,n}\bar T^{n}+\sum_{k=\left\lceil m\left(1+1/\log 2\right)\right\rceil}^{\left\lfloor\log_{2}\left(\frac{\bar T^{n}}{\phi(\bar T)}\right)\right\rfloor}2^{k+1}(k-m)^{m-1}\left(2^{-k}\bar T^{n}+\left(\frac{2^{-k}\bar T^{n}}{\phi(\bar T)}\right)^{\frac{m+n-1}{m+n}}\right)\label{eq:cor2eq2.1}  \\
 & \ll_{m,n}\sum_{k=0}^{\left\lfloor\log_{2}\left(\frac{\bar T^{n}}{\phi(\bar T)}\right)\right\rfloor}k^{m-1}\left(\bar T^{n}+2^{\frac{k}{m+n}}\left(\frac{\bar T^{n}}{\phi(\bar T)}\right)^{\frac{m+n-1}{m+n}}\right)\label{eq:cor2eq2.2},
\end{align}
where in (\ref{eq:cor2eq2.1}) we estimate $\# M\left(\bs{L},2^{-k},1/2,\bs{T}\right)$ with $\bar T^{n}$ for $k\leq \left\lfloor m\left(1+1/\log 2\right)\right\rfloor$. Note that $\bar T\geq 2$ ensures that (\ref{eq:cor2eq2.1})$\Rightarrow$(\ref{eq:cor2eq2.2}). The required result follows from (\ref{eq:cor2eq2.2}) combined with the trivial estimates $\sum_{k=0}^{K}k^{m-1}\ll_{m} K^{m}$ and $\sum_{k=0}^{K}k^{m-1}2^{\frac{k}{m+n}}\ll_{m,n} K^{m-1}2^{\frac{K}{m+n}}$.

\bibliographystyle{alpha}
\bibliography{References}

@Article{LV15,
 Author = {L{\^e}, Th{\'a}i Ho{\`a}ng and Vaaler, Jeffrey D.},
 Title = {Sums of products of fractional parts},
 FJournal = {Proceedings of the London Mathematical Society. Third Series},
 Journal = {Proc. Lond. Math. Soc. (3)},
 ISSN = {0024-6115},
 Volume = {111},
 Number = {3},
 Pages = {561--590},
 Year = {2015},
 Language = {English},
 DOI = {10.1112/plms/pdv038},
 zbMATH = {6488483},
 Zbl = {1327.11046}
}

@Article{HL22,
 Author = {Hardy, G. H. and Littlewood, J. E.},
 Title = {Some problems of diophantine approximation: {The} lattice-points of a right-angled triangle. {I}, {II}.},
 FJournal = {Proceedings of the London Mathematical Society. Second Series},
 Journal = {Proc. Lond. Math. Soc. (2)},
 ISSN = {0024-6115},
 Volume = {20},
 Pages = {15--36},
 Year = {1921},
 Language = {English},
 DOI = {10.1112/plms/s2-20.1.15},
 zbMATH = {2600869},
 JFM = {48.0197.07}
}

@Book{BHV20,
 Author = {Beresnevich, Victor and Haynes, Alan and Velani, Sanju},
 Title = {Sums of reciprocals of fractional parts and multiplicative {Diophantine} approximation},
 FSeries = {Memoirs of the American Mathematical Society},
 Series = {Mem. Am. Math. Soc.},
 ISSN = {0065-9266},
 Volume = {1276},
 ISBN = {978-1-4704-4095-4; 978-1-4704-5660-3},
 Year = {2020},
 Publisher = {Providence, RI: American Mathematical Society (AMS)},
 Language = {English},
 DOI = {10.1090/memo/1276},
 URL = {eprints.whiterose.ac.uk/114575/1/n_alpha_v2.pdf},
 zbMATH = {7196836},
 Zbl = {1445.11002}
}

@Article{Fre21-2,
 Author = {Fregoli, Reynold},
 Title = {On a counting theorem for weakly admissible lattices},
 FJournal = {IMRN. International Mathematics Research Notices},
 Journal = {Int. Math. Res. Not.},
 ISSN = {1073-7928},
 Volume = {2021},
 Number = {10},
 Pages = {7850--7884},
 Year = {2021},
 Language = {English},
 DOI = {10.1093/imrn/rnaa102},
 zbMATH = {7398537},
 Zbl = {1481.11095}
}

@Article{Cas50,
 Author = {Cassels, J. W. S.},
 Title = {Some metrical theorems in diophantine approximation. {I}, {III}},
 FJournal = {Proceedings of the Cambridge Philosophical Society},
 Journal = {Proc. Camb. Philos. Soc.},
 ISSN = {0008-1981},
 Volume = {46},
 Pages = {209--218, 219--225},
 Year = {1950},
 Language = {English},
 zbMATH = {3054216},
 Zbl = {0035.31902}
}

@Article{Sch60,
 Author = {Schmidt, Wolfgang},
 Title = {A metrical theorem in diophantine approximation},
 FJournal = {Canadian Journal of Mathematics},
 Journal = {Can. J. Math.},
 ISSN = {0008-414X},
 Volume = {12},
 Pages = {619--631},
 Year = {1960},
 Language = {English},
 DOI = {10.4153/CJM-1960-056-0},
 zbMATH = {3158848},
 Zbl = {0097.26205}
}

@Article{KSW17,
 Author = {Kleinbock, Dmitry and Shi, Ronggang and Weiss, Barak},
 Title = {Pointwise equidistribution with an error rate and with respect to unbounded functions},
 FJournal = {Mathematische Annalen},
 Journal = {Math. Ann.},
 ISSN = {0025-5831},
 Volume = {367},
 Number = {1-2},
 Pages = {857--879},
 Year = {2017},
 Language = {English},
 DOI = {10.1007/s00208-016-1404-3},
 zbMATH = {6690577},
 Zbl = {1417.37056}
}

@Book{Cas97,
 Author = {Cassels, J. W. S.},
 Title = {An introduction to the geometry of numbers.},
 Edition = {Repr. of the 1971 ed.},
 FSeries = {Classics in Mathematics},
 Series = {Class. Math.},
 ISSN = {1431-0821},
 ISBN = {3-540-61788-4},
 Year = {1997},
 Publisher = {Berlin: Springer},
 Language = {English},
 zbMATH = {971533},
 Zbl = {0866.11041}
}

@Article{KM99,
 Author = {Kleinbock, D. Y. and Margulis, G. A.},
 Title = {Logarithm laws for flows on homogeneous spaces},
 FJournal = {Inventiones Mathematicae},
 Journal = {Invent. Math.},
 ISSN = {0020-9910},
 Volume = {138},
 Number = {3},
 Pages = {451--494},
 Year = {1999},
 Language = {English},
 DOI = {10.1007/s002220050350},
 zbMATH = {1383858},
 Zbl = {0934.22016}
}

@Article{Wid17,
 Author = {Widmer, Martin},
 Title = {Asymptotic {Diophantine} approximation: the multiplicative case},
 FJournal = {The Ramanujan Journal},
 Journal = {Ramanujan J.},
 ISSN = {1382-4090},
 Volume = {43},
 Number = {1},
 Pages = {83--93},
 Year = {2017},
 Language = {English},
 DOI = {10.1007/s11139-016-9779-z},
 zbMATH = {6760608},
 Zbl = {1432.11074}
}

@Article{BW14,
 Author = {Barroero, Fabrizio and Widmer, Martin},
 Title = {Counting lattice points and {O}-minimal structures},
 FJournal = {IMRN. International Mathematics Research Notices},
 Journal = {Int. Math. Res. Not.},
 ISSN = {1073-7928},
 Volume = {2014},
 Number = {18},
 Pages = {4932--4957},
 Year = {2014},
 Language = {English},
 DOI = {10.1093/imrn/rnt102},
 zbMATH = {6365880},
 Zbl = {1315.11056}
}

@Article{Wal31,
 Author = {Walfisz, Arnold},
 Title = {{\"U}ber eine trigonometrische {Summe}},
 FJournal = {Journal of the London Mathematical Society},
 Journal = {J. Lond. Math. Soc.},
 ISSN = {0024-6107},
 Volume = {6},
 Pages = {169--172},
 Year = {1931},
 Language = {German},
 DOI = {10.1112/jlms/s1-6.3.169},
 zbMATH = {3001814},
 Zbl = {0002.18304}
}

@Article{Sch64,
 Author = {Schmidt, Wolfgang M.},
 Title = {Metrical theorems on fractional parts of sequences},
 FJournal = {Transactions of the American Mathematical Society},
 Journal = {Trans. Am. Math. Soc.},
 ISSN = {0002-9947},
 Volume = {110},
 Pages = {493--518},
 Year = {1964},
 Language = {English},
 DOI = {10.2307/1993695},
 zbMATH = {3317135},
 Zbl = {0199.09402}
}

@Article{Bec94,
 Author = {Beck, J{\'o}zsef},
 Title = {Probabilistic diophantine approximation. {I}: {Kronecker} sequences},
 FJournal = {Annals of Mathematics. Second Series},
 Journal = {Ann. Math. (2)},
 ISSN = {0003-486X},
 Volume = {140},
 Number = {2},
 Pages = {451--502},
 Year = {1994},
 Language = {English},
 DOI = {10.2307/2118607},
 zbMATH = {742789},
 Zbl = {0820.11045}
}

@Article{HL12,
 Author = {Hofer, Roswitha and Larcher, Gerhard},
 Title = {Metrical results on the discrepancy of {Halton}-{Kronecker} sequences},
 FJournal = {Mathematische Zeitschrift},
 Journal = {Math. Z.},
 ISSN = {0025-5874},
 Volume = {271},
 Number = {1-2},
 Pages = {1--11},
 Year = {2012},
 Language = {English},
 DOI = {10.1007/s00209-011-0848-0},
 zbMATH = {6050606},
 Zbl = {1257.11073}
}

@Article{Wid18,
 Author = {Widmer, Martin},
 Title = {Weak admissibility, primitivity, o-minimality, and {Diophantine} approximation},
 FJournal = {Mathematika},
 Journal = {Mathematika},
 ISSN = {0025-5793},
 Volume = {64},
 Number = {2},
 Pages = {475--496},
 Year = {2018},
 Language = {English},
 DOI = {10.1112/S0025579318000049},
 zbMATH = {6908360},
 Zbl = {1401.11113}
}

@article{Kru64,
author="A. H. Kruse",
title="Estimates of {$\sum_{k=1}^{N}k^{-s}\langle kx\rangle^{-t}$}",
journal="Trans. Amer. Math. Soc.",
volume="110",
pages="493-518",
year="1964"}

@article{BL17,
author="V. Beresnevich and N. Leong",
title="Sums of reciprocals and the three distance theorem",
journal="arXiv:1712.03758 [math.NT]",
year="2017"}

@article{Bug14,
author="Y. Bugeaud",
title="Around The {L}ittlewood conjecture in {D}iophantine approximation",
journal="Publications math{\'e}matiques de Besan\c{c}on",
volume="2014",
number="No. 1",
pages="5-18",
year="2014"}

@Article{Skr90,
 Author = {Skriganov, M. M.},
 Title = {Lattices in algebraic number fields and uniform distribution mod 1},
 FJournal = {Leningrad Mathematical Journal},
 Journal = {Leningr. Math. J.},
 ISSN = {1048-9924},
 Volume = {1},
 Number = {2},
 Pages = {535--558},
 Year = {1990},
 Language = {English},
 zbMATH = {4175046},
 Zbl = {0714.11045}
}

@Article{CT19,
 Author = {Chow, Sam and Technau, Niclas},
 Title = {Higher-rank {Bohr} sets and multiplicative {Diophantine} approximation},
 FJournal = {Compositio Mathematica},
 Journal = {Compos. Math.},
 ISSN = {0010-437X},
 Volume = {155},
 Number = {11},
 Pages = {2214--2233},
 Year = {2019},
 Language = {English},
 DOI = {10.1112/S0010437X19007589},
 zbMATH = {7118441},
 Zbl = {1428.11133}
}

\end{document}